\pdfoutput=1
\RequirePackage{ifpdf}
\ifpdf 
\documentclass[pdftex]{sigma}
\else
\documentclass{sigma}
\fi

\usepackage{tikz}

\numberwithin{equation}{section}

\newtheorem{Theorem}{Theorem}[section]
\newtheorem*{Theorem*}{Theorem}

\newtheorem{Lemma}[Theorem]{Lemma}
\newtheorem{Proposition}[Theorem]{Proposition}
\newtheorem{Conjecture}[Theorem]{Conjecture}
 { \theoremstyle{definition}
\newtheorem{Definition}[Theorem]{Definition}

\newtheorem{Example}[Theorem]{Example}
\newtheorem{Remark}[Theorem]{Remark} }

\newcommand{\C}{{\mathbb C}}
\newcommand{\Q}{{\mathbb Q}}
\newcommand{\R}{{\mathbb R}}
\newcommand{\Z}{{\mathbb Z}}

\newcommand{\geh}{\mathfrak{g}}

\newcommand{\ep}{\epsilon}

\newcommand{\La}{\Lambda}

\newcommand{\op}{\oplus}
\newcommand{\Op}{\bigoplus}

\newcommand{\pair}[1]{\langle{#1}\rangle}

\begin{document}
\allowdisplaybreaks

\newcommand{\arXivNumber}{2410.02286}

\renewcommand{\PaperNumber}{028}

\FirstPageHeading

\ShortArticleName{Note on Exponents Associated with Y-Systems}

\ArticleName{Note on Exponents Associated with Y-Systems}

\Author{Ryo TAKENAKA}

\AuthorNameForHeading{R.~Takenaka}

\Address{Department of Mathematics, Osaka Metropolitan University, Osaka 558-8585, Japan}
\Email{\href{mailto:r.takenaka0419@gmail.com}{r.takenaka0419@gmail.com}}

\ArticleDates{Received October 14, 2024, in final form April 16, 2025; Published online April 24, 2025}

\Abstract{Let $(X_n,\ell)$ be the pair consisting of the Dynkin diagram of finite type $X_n$ and a positive integer $\ell\geq2$, called the level. Then we obtain the Y-system, which is the set of algebraic relations associated with this pair. Related to the Y-system, a sequence of integers called exponents is defined through a quiver derived from the pair $(X_n,\ell)$. Mizuno provided conjectured formulas for the exponents associated with Y-systems in [Mizuno Y., \textit{SIGMA} \textbf{16} (2020), 028, 42~pages, arXiv:1812.05863]. In this paper, we study the exponents associated with level 2 Y-systems for classical Dynkin types. As a result, we present proofs of Mizuno's conjecture for $(B_n,2)$ and $(D_n,2)$, and give a~reformulation for $(C_n,2)$.}


\Classification{13F60; 17B22; 81R10}

\section{Introduction}

Cluster algebras introduced by Fomin and Zelevinsky in the early 2000s \cite{FZ,FZ2} have played a significant role in many areas of mathematics.
Due to their broadly applicable structure, they have been studied in various fields such as representation theory, combinatorics, algebraic geometry, and number theory, among others.
In particular, we focus on recent research by Mizuno \cite{M}, which calculates a certain class of Y-systems.
The prototypes of Y-systems were introduced by Zamolodchikov around 90s in the study of two-dimensional integrable filed theory \cite{Z}.
Then, some generalization were introduced by Ravanini--Tateo--Valleriani, Kuniba--Nakanishi--Suzuki, and others \cite{KNS0, RTV}.
His research has revealed a connection between cluster algebras and the representation theory of affine Lie algebras.
This connection was established through the coincidence of hypergeometric series, known as partition $q$-series and the string functions of affine Lie algebras.
Partition $q$-series, which is defined by Kato and Terashima \cite{KT}, is a certain $q$-series associated with a mutation sequence.
Mizuno has shown that the asymptotics of this partition $q$-series is calculated by matrices determined by the mutation sequence and the corresponding cluster algebra.
Furthermore, we are able to verify that the limit of the partition $q$-series coincides with the asymptotic dimension of an irreducible representation of the affine Lie algebra found by Kac and Peterson \cite{KP}.

We explore the intersection of these two theories and examine the sequence of integers called exponents, which are introduced from the theory of cluster algebras.
First, we review the definitions of quiver mutations and mutation loops provided by pairs of quivers and their mutations.
Then we define a cluster transformation of variables associated to the mutation loop.
A~Dynkin quiver is determined by a Dynkin diagram and the integer $\ell$ greater than 2, referred to as a~level~\mbox{\cite{IIKKN, IIKKN2}}.
After explaining the property of Dynkin quivers, we will specifically construct cases for level 2.
We review the Y-system, which plays an important role in this paper.
Y-systems have been extensively studied as important objects in mathematical physics and integrable systems.
The Y-system is formulated through simultaneous nonlinear difference equations, which describe the relationship between the Y-functions associated with integrable models.
In this paper, we treat the Y-system associated with RSOS models introduce by Kuniba--Nakanishi \cite{KN}.
Such Y-system is a set of algebraic relations defined for a pair consisting of the Dynkin diagram and the integer $\ell$.
We follow \cite{KNS} on the Y-system.
The periodicities of the Y-systems were conjectured.
Associated to the introduction of cluster algebra, connection between Zamolodchikov's Y-systems was established by Fomin and Zelevinsky.
A similar connection was established by Keller and Inoue--Iyama--Kuniba--Nakanishi--Suzuki for more general Y-systems.
Then the periodicity has been proved by using the cluster algebraic method in \cite{IIKKN, IIKKN2, K}.
Therefore, we have that the cluster transformation for mutation loops exhibits periodicity, which allow us to define exponents.
Due to the periodicity, we find that eigenvalues of the Jacobian for the cluster transformation are roots of unity.
Therefore, we are able to characterize eigenvalues using a sequence of nonnegative integers.
We call such sequences exponents of Dynkin quivers.

For the exponents associated with these Y-systems, Mizuno provided Conjecture \ref{MConj} which gives conjectural formulas based on the corresponding root systems.
Furthermore, he proved for these conjectural formulas in the cases $(A_1,\ell)$ and $(A_n,2)$ (cf.\ \cite{M}).
In this study, we investigate level 2 Dynkin quivers associated with other classical type Dynkin diagrams.
We prove Mizuno's conjecture for $(B_n,2)$ and $(D_n,2)$, and reduce it for $(C_n,2)$ to Conjecture \ref{Csol}.

	\section{Preliminary}

	\subsection{Quiver and Y-seed mutations}
A quiver is a directed graph with vertices ${\bf I}=\{1,\dots,N\}$ that may have multiple edges.
In this paper, we assume that quivers do not have 1,2-cycles.
We define quiver mutations as follows.
		\begin{Definition}
		Let $Q$ be a quiver and let $k$ be a vertex of $Q$.
		The quiver mutation $\mu_k^\prime$ is a~transformation of quiver defined by the following steps:
			\begin{enumerate}\itemsep=0pt
			\item For each length two path $i\rightarrow k\rightarrow j$, add the new arrow $i\rightarrow j$.
			\item Reverse all arrows incident to the vertex $k$.
			\item Remove all 2-cycles.
			\end{enumerate}
		\end{Definition}
For instance, we have the following transition.
	
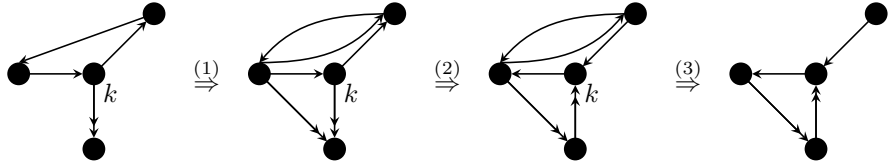
\begin{Example}\label{quiver}
			\[\begin{tikzpicture}
				\fill[black](0,0)circle(1ex)node[below right]{$k$};
				\fill[black](-1,0)circle(1ex);
				\fill[black](0.8,0.8)circle(1ex);
				\fill[black](0,-1)circle(1ex);
				\draw[semithick][->,>=stealth](0,0)--(0.7,0.7);
				\draw[semithick][->,>=stealth](-1,0)--(-0.15,0);
				\draw[semithick][->,>=stealth](0,0)--(0,-0.85);
				\draw[semithick][->,>=stealth](0,0)--(0,-0.7);
				\draw[semithick][->,>=stealth](0.8,0.8)--(-1,0.15);
				\node at(1.5,0){$\overset{{\rm (1)}}{\Rightarrow}$};
				\fill[black](2.2,0)circle(1ex);
				\fill[black](3.2,0)circle(1ex)node[below right]{$k$};
				\fill[black](3.2,-1)circle(1ex);
				\fill[black](4,0.8)circle(1ex);
				\draw[semithick][->,>=stealth](3.2,0)--(3.2,-0.85);
				\draw[semithick][->,>=stealth](3.2,0)--(3.2,-0.7);
				\draw[semithick][->,>=stealth](3.2,0)--(3.9,0.7);
				\draw[semithick][->,>=stealth](2.2,0)--(3.05,0);
				\draw[semithick][->,>=stealth](2.2,0)--(3.1,-0.9);
				\draw[semithick][->,>=stealth](2.2,0)--(3,-0.8);
				\draw[semithick][->,>=stealth](4,0.8)to[out=180,in=50](2.2,0.15);
				\draw[semithick][->,>=stealth](2.2,0.15)to[out=0,in=230](3.85,0.8);
				\node at(4.7,0){$\overset{\rm (2)}{\Rightarrow}$};
				\fill[black](5.4,0)circle(1ex);
				\fill[black](6.4,0)circle(1ex)node[below right]{$k$};
				\fill[black](6.4,-1)circle(1ex);
				\fill[black](7.2,0.8)circle(1ex);
				\draw[semithick][->,>=stealth](6.4,0)--(5.55,0);
				\draw[semithick][->,>=stealth](6.4,-1)--(6.4,-0.15);
				\draw[semithick][->,>=stealth](6.4,-1)--(6.4,-0.3);
				\draw[semithick][->,>=stealth](7.2,0.8)--(6.5,0.1);
				\draw[semithick][->,>=stealth](7.2,0.8)to[out=180,in=50](5.4,0.15);
				\draw[semithick][->,>=stealth](5.4,0.15)to[out=0,in=230](7.05,0.8);
				\draw[semithick][->,>=stealth](5.4,0)--(6.3,-0.9);
				\draw[semithick][->,>=stealth](5.4,0)--(6.2,-0.8);
				\node at(7.9,0){$\overset{\rm (3)}{\Rightarrow}$};
				\fill[black](8.6,0)circle(1ex);
				\fill[black](9.6,0)circle(1ex);
				\fill[black](9.6,-1)circle(1ex);
				\fill[black](10.4,0.8)circle(1ex);
				\draw[semithick][->,>=stealth](9.6,0)--(8.75,0);
				\draw[semithick][->,>=stealth](8.6,0)--(9.5,-0.9);
				\draw[semithick][->,>=stealth](8.6,0)--(9.4,-0.8);
				\draw[semithick][->,>=stealth](9.6,-1)--(9.6,-0.15);
				\draw[semithick][->,>=stealth](9.6,-1)--(9.6,-0.25);
				\draw[semithick][->,>=stealth](10.4,0.8)--(9.7,0.1);
				\end{tikzpicture}\]
	\end{Example}
Consider a quiver $Q$ with vertices {\bf I}.
Let $Q_{i,j}$ be the number of arrows from $i$ to $j$ and $\nu$ be a permutation of $\{1,\dots,N\}$.
Then the action of $\nu$ on $Q$ is defined by $\nu(Q)_{i,j}=Q_{\nu^{-1}(i),\nu^{-1}(j)}$.
Set ${\bf m}=(m_1,\dots,m_T)$ to be a sequence of ${\bf I}$.
Then we have the transformation \[Q(0)\overset{\mu_{m_1}^\prime}{\longrightarrow}Q(1)\overset{\mu_{m_2}^\prime}{\longrightarrow}\cdots\overset{\mu_{m_T}^\prime}{\longrightarrow}Q(T)\overset{\nu}{\longrightarrow}\nu(Q(T)),\]
where $Q(0)$ stands for $Q$.
The triple $\gamma=(Q,{\bf m},\nu)$ is called a mutation loop if $Q(0)=\nu(Q(T))$.

Now we see the definition of Y-seed mutations.
Consider the set for variables $y_1,\dots,y_N$ given~by
		\[\mathcal{F}=\left\{\left.\frac{f(y_1,\dots,y_N)}{g(y_1,\dots,y_N)}\right|f,g\in\Q_{\geq0}[y_1,\dots,y_N]\setminus\{0\}\right\}.\]
This is closed under the usual multiplication and addition and called universal semifield in the variables $y_1,\dots,y_N$.
A Y-seed is a pair $(Q,Y)$ for a quiver $Q$ and \smash{$Y=(Y_1,\dots,Y_N)\in\mathcal{F}^N$}.
For given pair, Y-seed mutations are defined as follows.
		\begin{Definition}
		Let $k$ be vertex of $Q$.
		The Y-seed mutation $\mu_k$ is a transformation of the pair~$(Q,Y)$ obtained by extending the definition of the quiver mutation so that
			\[\mu_k(Q)=\mu_k^\prime(Q),\qquad\mu_k(Y)_i=\begin{cases}
				Y_k^{-1}&\text{if }i=k,\\
				Y_i\bigl(Y_k^{-1}+1\bigr)^{-Q_{k,i}}&\text{if }i\neq k,\  Q_{k,i}\geq0,\\
				Y_i(Y_k+1)^{Q_{i,k}}&\text{if }i\neq k,\  Q_{i,k}\geq0.
			\end{cases}\]
		\end{Definition}

		\begin{Example}
		For the quiver given in Example \ref{quiver}, set variables $(y_1,y_2,y_3,y_4)$ as follows:
			
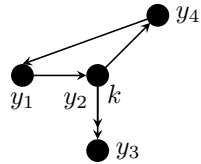
\[\begin{tikzpicture}
				\fill[black](0,0)circle(1ex)node[below right]{$k$};
				\node at (0,0)[below left]{$\overset{}{y_2}$};
				\fill[black](-1,0)circle(1ex)node[below]{$\overset{}{y_1}$};
				\fill[black](0.8,0.8)circle(1ex)node[right]{$\ y_4$};
				\fill[black](0,-1)circle(1ex)node[right]{$\ y_3$};
				\draw[semithick][->,>=stealth](0,0)--(0.7,0.7);
				\draw[semithick][->,>=stealth](-1,0)--(-0.15,0);
				\draw[semithick][->,>=stealth](0,0)--(0,-0.85);
				\draw[semithick][->,>=stealth](0,0)--(0,-0.7);
				\draw[semithick][->,>=stealth](0.8,0.8)--(-1,0.15);
			\end{tikzpicture}\]
		Then we obtain the Y-seed mutation $\mu_k$ by
			\[\mu_k\begin{pmatrix}
							y_1\\
							y_2\\
							y_3\\
							y_4
							\end{pmatrix}=
			\begin{pmatrix}
							y_1(y_2+1)\\
							y_2^{-1}\\
							y_3\bigl(y_2^{-1}+1\bigr)^{-2}\\
							y_4\bigl(y_2^{-1}+1\bigr)^{-1}
			\end{pmatrix}.\]
		\end{Example}

We also define the action of $\nu$ on $(Q,Y)$ by $\nu(Q,Y)=(\nu(Q),\nu(Y))$ where $\nu(Y_i)=Y_{\nu^{-1}(i)}$.
Let $\gamma=(Q,{\bf m},\nu)$ be a mutation loop and $(Q,Y)$ be a Y-seed.
Then we have the transformation
\[(Q(0),Y(0))\overset{\mu_{m_1}}{\longrightarrow}\cdots\overset{\mu_{m_T}}{\longrightarrow}(Q(T),Y(T))\overset{\nu}{\longrightarrow}(\nu(Q(T),\nu(Y(T))),\]
where $Y(0)$ stands for $Y$.
Although $Q(0)=\nu(Q(T))$ holds from the definition of the mutation loop, we have $Y(0)\neq\nu(Y(T))$ in general.
We denote $\nu(Y(T))$ by $\mu_\gamma(Y)$ and call it the cluster transformation of $\gamma$.
By definition, we find that $\mu_\gamma(Y)\in\mathcal{F}^N$.

	\subsection[Dynkin quiver Q(X\_n,ell) and mutation loop on Q(X\_n,ell)]{Dynkin quiver $\boldsymbol{Q(X_n,\ell)}$ and mutation loop on $\boldsymbol{Q(X_n,\ell)}$}

For any Dynkin diagram of type $X_n$ and positive integer $\ell\geq2$ called level, Inoue, Iyama, Keller, Kuniba and Nakanishi introduce a Dynkin quiver $Q=Q(X_n,\ell)$ and a mutation loop $\gamma=\gamma(X_n,\ell)$ on $Q$ in \cite{IIKKN, IIKKN2}.
Correspond to the type of Dynkin diagram, Dynkin quiver consists of the following vertices:
	\begin{itemize}\itemsep=0pt
	\item $X=A,D,E$: black vertices labeled with $+$ or $-$.
	\item $X=B,C,F$: black vertices labeled with $+$ or $-$, and white vertices labeled with $+$ or $-$.
	\item $X=G$: black vertices labeled with $+$ or $-$, and white vertices labeled with~I,~II,~III,~IV,~V or VI.
	\end{itemize}
Associated with these vertices, we set the following subset of ${\bf I}$:
	\begin{itemize}\itemsep=0pt
	\item ${\bf I}_{+,-}^\circ$: the set corresponding to white vertices labeled with $+$ or $-$.
	\item ${\bf I}_{+,-}^\bullet$: the set corresponding to black vertices labeled with $-$ or $-$.
	\item ${\bf I}_{\rm I,II}^\circ$: the set corresponding to white vertices labeled with I or II.
	\end{itemize}
Furthermore, the permutation $\nu$ on $Q$ is given as a folding induced by symmetry of white vertices with respect to black vertices.
We define the two compositions of mutation
		\[\mu_+=\prod_{k\in{\bf I}_+^{\bullet}\cup{\bf I}_+^{\circ}\cup{\bf I}_{\rm I}^\circ}\mu_k,\qquad \mu_-=\prod_{k\in{\bf I}_-^{\bullet}\cup{\bf I}_{\rm II}^\circ}\mu_k.\]
In \cite{IIKKN,IIKKN2}, the following property is shown.
		\begin{Lemma}
		Let $Q=Q(X_n,\ell)$ and $\nu$ be the permutation of vertices in $Q$.
		Consider the transformation
			\[Q\overset{\mu_+}{\longrightarrow}Q^\prime\overset{\mu_-}{\longrightarrow}Q^{\prime\prime}.\]
		Then $Q^\prime$ and $Q^{\prime\prime}$ are independent on orders of the mutations in $\mu_+$ and $\mu_-$.
		Furthermore, we have $Q=\nu(Q^{\prime\prime})$.
		\end{Lemma}
Therefore, the lemma suggests that $\gamma(X_n,\ell)=\bigl(Q(X_n,\ell),\bigl({\bf I}_+^{\bullet}\cup{\bf I}_+^{\circ}\cup{\bf I}_{\rm I}^\circ,{\bf I}_-^{\bullet}\cup{\bf I}_{\rm II}^\circ\bigr),\nu\bigr)$ is a~mutation loop.
Note that $\bigl({\bf I}_+^{\bullet}\cup{\bf I}_+^{\circ}\cup{\bf I}_{\rm I}^\circ,{\bf I}_-^{\bullet}\cup{\bf I}_{\rm II}^\circ\bigr)$ stands for a sequence obtained by lining up elements in ${\bf I}_+^\bullet\cup{\bf I}_+^\circ\cup{\bf I}_{\rm I}^\circ$ and ${\bf I}_-^\bullet\cup{\bf I}_{\rm II}^\circ$.
The cluster transformation $\mu_\gamma=\nu\circ\mu_-\circ\mu_+$ has the following periodicity.
		\begin{Theorem}[{\cite{IIKKN, IIKKN2, K}}]
		For the mutation loop $\gamma=\gamma(X_n,\ell)$, we have
			\[\mu_\gamma^{t(\ell+h^\vee)}(Y)=Y\]
		on \smash{$\mathcal{F}^N$}, where $h^\vee$ is the dual Coxeter number of $X_n$ and $t$ is given by
			\[t=\begin{cases}
				1,&X_n=A_n,D_n,E_{6,7,8},\\
				2,&X_n=B_n,C_n,F_4,\\
				3,&X_n=G_2.
			\end{cases}\]
		\end{Theorem}

If we restrict the level to 2, then we obtain Dynkin quivers $Q(X_n,2)$ and permutations $\nu$ as in Figure~\ref{Dq}.
Note that the permutation $\nu$ of vertices is expressed by arrow $\twoheadrightarrow$ in the figure.
Namely, Dynkin quivers for simply-laced types have the identity as $\nu$.

	\begin{figure}[ht]\centering
\begin{tikzpicture}
				\node at(-1.5,0){$(A_n,2)$};
				\fill[black](0,0)circle(1ex)node[above]{$-$};
				\fill[black](0.7,0)circle(1ex)node[above]{$+$};
				\fill[black](2.45,0)circle(1ex)node[above]{$-$};
				\fill[black](3.15,0)circle(1ex)node[above]{$+$};
				\draw[semithick][->,>=stealth](0,0)--(0.55,0);
				\draw[semithick][->,>=stealth](1.2,0)--(0.85,0);
				\draw[dashed](0.7,0)--(2.45,0);
				\draw[semithick][-](1.95,0)--(2.3,0);
				\draw[semithick][->,>=stealth](2.45,0)--(3,0);
				\draw[dashed](3.15,0)--(3.9,0);
				\draw[-](1.95,-1)to[out=90,in=270](0,-0.2);
				\draw[-](1.95,-1)to[out=90,in=270](3.9,-0.2);
				\node at(1.95,-1.2){$n$};
				\node at(5.65,0){$(B_n,2)$};
				\draw(7.65,0)circle(1ex)node[above]{$-$};
				\draw(8.35,0)circle(1ex)node[above]{$+$};
				\draw(9.05,0)circle(1ex)node[above]{$-$};
				\fill[black](9.75,0)circle(1ex)node[above right]{$-$};
				\fill[black](9.75,0.7)circle(1ex)node[above right]{$+$};
				\fill[black](9.75,-0.7)circle(1ex)node[above right]{$+$};
				\draw(10.45,0)circle(1ex)node[above]{$+$};
				\draw(11.15,0)circle(1ex)node[above]{$-$};
				\draw(11.85,0)circle(1ex)node[above]{$+$};
				\draw[dashed](6.9,0)--(7.5,0);
				\draw[semithick][->,>=stealth](7.8,0)--(8.2,0);
				\draw[semithick][->,>=stealth](8.9,0)--(8.5,0);
				\draw[semithick][->,>=stealth](9.75,0)--(9.2,0);
				\draw[semithick][->,>=stealth](9.15,0.1)--(9.65,0.6);
				\draw[semithick][->,>=stealth](9.15,-0.1)--(9.65,-0.6);
				\draw[semithick][->,>=stealth](9.75,0.7)--(9.75,0.15);
				\draw[semithick][->,>=stealth](9.75,-0.7)--(9.75,-0.15);
				\draw[semithick][->,>=stealth](9.75,0)--(10.3,0);
				\draw[semithick][->,>=stealth](11,0)--(10.6,0);
				\draw[semithick][->,>=stealth](11.3,0)--(11.7,0);
				\draw[dashed](12,0)--(12.6,0);
				\draw[->>](9.75,-1.4)to[out=180,in=270](7.65,-0.2);
				\draw[->>](9.75,-1.2)to[out=180,in=270](8.35,-0.2);
				\draw[->>](9.75,-1)to[out=180,in=270](9.05,-0.2);
				\draw[->>](9.75,-1.4)to[out=0,in=270](11.85,-0.2);
				\draw[->>](9.75,-1.2)to[out=0,in=270](11.15,-0.2);
				\draw[->>](9.75,-1)to[out=0,in=270](10.45,-0.2);
				\node at(9.75,-1.6){$\nu$};
				\draw[-](8,0.7)to[out=270,in=90](9.05,0.35);
				\draw[-](8,0.7)to[out=270,in=90](6.95,0.35);
				\node at(8,0.8){$n-1$};
				\draw[-](11.5,0.7)to[out=270,in=90](10.45,0.35);
				\draw[-](11.5,0.7)to[out=270,in=90](12.55,0.35);
				\node at(11.5,0.8){$n-1$};
				\node at(-1.5,-3){$(C_n,2)$};
				\fill[black](0.5,-2.3)circle(1ex)node[above right]{$-$};
				\fill[black](0.5,-3)circle(1ex)node[above right]{$+$};
				\fill[black](0.5,-3.7)circle(1ex)node[above right]{$-$};
				\fill[black](1.2,-2.3)circle(1ex)node[above right]{$+$};
				\fill[black](1.2,-3)circle(1ex)node[above right]{$-$};
				\fill[black](1.2,-3.7)circle(1ex)node[above right]{$+$};
				\fill[black](1.9,-2.3)circle(1ex)node[above right]{$-$};
				\fill[black](1.9,-3)circle(1ex)node[above right]{$+$};
				\fill[black](1.9,-3.7)circle(1ex)node[above right]{$-$};
				\fill[black](2.6,-2.3)circle(1ex)node[above right]{$+$};
				\fill[black](2.6,-3)circle(1ex)node[above right]{$-$};
				\fill[black](2.6,-3.7)circle(1ex)node[above right]{$+$};
				\draw(3.6,-2.65)circle(1ex)node[above right]{$+$};
				\draw(3.6,-3.35)circle(1ex)node[below right]{$-$};
				\draw[dashed](0.35,-2.3)--(-0.15,-2.3);
				\draw[dashed](0.35,-3)--(-0.15,-3);
				\draw[dashed](0.35,-3.7)--(-0.15,-3.7);
				\draw[semithick][->,>=stealth](0.5,-3)--(0.5,-2.45);
				\draw[semithick][->,>=stealth](0.5,-3)--(0.5,-3.55);
				\draw[semithick][->,>=stealth](0.5,-2.3)--(1.05,-2.3);
				\draw[semithick][->,>=stealth](1.2,-3)--(0.65,-3);
				\draw[semithick][->,>=stealth](0.3,-3.7)--(1.05,-3.7);
				\draw[semithick][->,>=stealth](1.2,-2.3)--(1.2,-2.85);
				\draw[semithick][->,>=stealth](1.2,-3.7)--(1.2,-3.15);
				\draw[semithick][->,>=stealth](1.9,-2.3)--(1.35,-2.3);
				\draw[semithick][->,>=stealth](1.2,-3)--(1.75,-3);
				\draw[semithick][->,>=stealth](1.9,-3.7)--(1.35,-3.7);
				\draw[semithick][->,>=stealth](1.9,-3)--(1.9,-2.45);
				\draw[semithick][->,>=stealth](1.9,-3)--(1.9,-3.55);
				\draw[semithick][->,>=stealth](1.9,-2.3)--(2.45,-2.3);
				\draw[semithick][->,>=stealth](2.6,-3)--(2.05,-3);
				\draw[semithick][->,>=stealth](1.9,-3.7)--(2.45,-3.7);
				\draw[semithick][->,>=stealth](2.6,-2.3)--(2.6,-2.85);
				\draw[semithick][->,>=stealth](2.6,-3.7)--(2.6,-3.15);
				\draw[semithick][->,>=stealth](2.6,-3)--(3.45,-2.65);
				\draw[semithick][->,>=stealth](2.6,-3)--(3.45,-3.35);
				\draw[semithick][->,>=stealth](3.5,-3.45)--(2.75,-3.7);
				\draw[-,line width=5pt,white](3.4,-3.15)--(2.8,-2.4);
				\draw[semithick][->,>=stealth](3.5,-3.25)--(2.75,-2.3);
				\draw[->>](4.1,-3)to[out=90,in=330](3.8,-2.65);
				\draw[->>](4.1,-3)to[out=270,in=30](3.8,-3.35);
				\node at(4.3,-3){$\nu$};
				\draw[-](1.2,-4.4)to[out=90,in=270](-0.2,-3.9);
				\draw[-](1.2,-4.4)to[out=90,in=270](2.6,-3.9);
				\node at(1.2,-4.5){$n-1$};
				\node at(5.65,-3){$(D_n,2)$};
				\fill[black](7.75,-3)circle(1ex)node[above]{$-$};
				\fill[black](8.45,-3)circle(1ex)node[above]{$+$};
				\fill[black](10.2,-3)circle(1ex)node[above]{$-$};
				\fill[black](10.9,-3)circle(1ex)node[above]{$+$};
				\fill[black](11.6,-2.3)circle(1ex)node[right]{$-$};
				\fill[black](11.6,-3.7)circle(1ex)node[right]{$-$};
				\draw[dashed](7.75,-3)--(7,-3);
				\draw[semithick][->,>=stealth](7.9,-3)--(8.3,-3);
				\draw[semithick][->,>=stealth](8.95,-3)--(8.6,-3);
				\draw[dashed](8.45,-3)--(10.2,-3);
				\draw[semithick][-](9.7,-3)--(10.05,-3);
				\draw[semithick][->,>=stealth](10.2,-3)--(10.75,-3);
				\draw[semithick][->,>=stealth](11.6,-2.3)--(11,-2.9);
				\draw[semithick][->,>=stealth](11.6,-3.7)--(11,-3.1);
				\draw[-](8.95,-4)to[out=90,in=270](7,-3.2);
				\draw[-](8.95,-4)to[out=90,in=270](10.9,-3.2);
				\node at(8.95,-4.2){$n-2$};
				\node at(-1.5,-6){$(E_6,2)$};
				\fill[black](0,-6)circle(1ex)node[below]{$\overset{1}{-}$};
				\fill[black](0.7,-6)circle(1ex)node[below]{$\overset{2}{+}$};
				\fill[black](1.4,-6)circle(1ex)node[below]{$\overset{3}{-}$};
				\fill[black](2.1,-6)circle(1ex)node[below]{$\overset{5}{+}$};
				\fill[black](2.8,-6)circle(1ex)node[below]{$\overset{6}{-}$};
				\fill[black](1.4,-5.3)circle(1ex)node[right]{$\overset{4}{+}$};
				\draw[semithick][->,>=stealth](0,-6)--(0.55,-6);
				\draw[semithick][->,>=stealth](1.4,-6)--(0.85,-6);
				\draw[semithick][->,>=stealth](1.4,-6)--(1.4,-5.45);
				\draw[semithick][->,>=stealth](1.4,-6)--(1.95,-6);
				\draw[semithick][->,>=stealth](2.8,-6)--(2.25,-6);
				\node at(5.65,-6){$(E_7,2)$};
				\fill[black](7.15,-6)circle(1ex)node[below]{$\overset{1}{-}$};
				\fill[black](7.85,-6)circle(1ex)node[below]{$\overset{2}{+}$};
				\fill[black](8.55,-6)circle(1ex)node[below]{$\overset{3}{-}$};
				\fill[black](8.55,-5.3)circle(1ex)node[right]{$\overset{7}{+}$};
				\fill[black](9.25,-6)circle(1ex)node[below]{$\overset{4}{+}$};
				\fill[black](9.95,-6)circle(1ex)node[below]{$\overset{5}{-}$};
				\fill[black](10.65,-6)circle(1ex)node[below]{$\overset{6}{+}$};
				\draw[semithick][->,>=stealth](7.15,-6)--(7.7,-6);
				\draw[semithick][->,>=stealth](8.55,-6)--(8,-6);
				\draw[semithick][->,>=stealth](8.55,-6)--(8.55,-5.45);
				\draw[semithick][->,>=stealth](8.55,-6)--(9.15,-6);
				\draw[semithick][->,>=stealth](9.95,-6)--(9.4,-6);
				\draw[semithick][->,>=stealth](9.95,-6)--(10.5,-6);
				\node at(-1.5,-9){$(E_8,2)$};
				\fill[black](0,-9)circle(1ex)node[below]{$\overset{1}{-}$};
				\fill[black](0.7,-9)circle(1ex)node[below]{$\overset{2}{+}$};
				\fill[black](1.4,-9)circle(1ex)node[below]{$\overset{3}{-}$};
				\fill[black](2.1,-9)circle(1ex)node[below]{$\overset{4}{+}$};
				\fill[black](2.8,-9)circle(1ex)node[below]{$\overset{5}{-}$};
				\fill[black](2.8,-8.3)circle(1ex)node[right]{$\overset{8}{+}$};
				\fill[black](3.5,-9)circle(1ex)node[below]{$\overset{6}{+}$};
				\fill[black](4.2,-9)circle(1ex)node[below]{$\overset{7}{-}$};
				\draw[semithick][->,>=stealth](0,-9)--(0.55,-9);
				\draw[semithick][->,>=stealth](1.4,-9)--(0.85,-9);
				\draw[semithick][->,>=stealth](1.4,-9)--(1.95,-9);
				\draw[semithick][->,>=stealth](2.8,-9)--(2.25,-9);
				\draw[semithick][->,>=stealth](2.8,-9)--(2.8,-8.45);
				\draw[semithick][->,>=stealth](2.8,-9)--(3.35,-9);
				\draw[semithick][->,>=stealth](4.2,-9)--(3.65,-9);
				\node at(5.65,-9){$(F_4,2)$};
				\draw(7.15,-9)circle(1ex)node[above]{$+$};
				\draw(7.85,-9)circle(1ex)node[above]{$-$};
				\fill[black](8.55,-9)circle(1ex)node[below right]{$-$};
				\fill[black](8.55,-8.3)circle(1ex)node[below right]{$+$};
				\fill[black](8.55,-9.7)circle(1ex)node[below right]{$+$};
				\fill[black](9.25,-8.65)circle(1ex)node[right]{$+$};
				\fill[black](9.25,-7.95)circle(1ex)node[right]{$-$};
				\fill[black](9.25,-9.35)circle(1ex)node[right]{$-$};
				\draw(9.95,-9)circle(1ex)node[above]{$+$};
				\draw(10.65,-9)circle(1ex)node[above]{$-$};
				\draw[semithick][->,>=stealth](7.7,-9)--(7.3,-9);
				\draw[semithick][->,>=stealth](7.95,-8.9)--(8.45,-8.4);
				\draw[semithick][->,>=stealth](7.95,-9.1)--(8.45,-9.6);
				\draw[semithick][->,>=stealth](8.55,-9)--(8,-9);
				\draw[semithick][->,>=stealth](8.55,-8.3)--(8.55,-8.85);
				\draw[semithick][->,>=stealth](8.55,-9.7)--(8.55,-9.15);
				\draw[semithick][->,>=stealth](8.55,-9)--(9.1,-8.65);
				\draw[semithick][->,>=stealth](9.25,-8.65)--(9.25,-8.1);
				\draw[semithick][->,>=stealth](9.25,-8.65)--(9.25,-9.2);
				\draw[semithick][->,>=stealth](9.25,-7.95)--(8.7,-8.3);
				\draw[semithick][->,>=stealth](9.25,-9.35)--(8.7,-9.7);
				\draw[-,line width=4pt,white](8.8,-9)--(9.7,-9);
				\draw[semithick][->,>=stealth](8.55,-9)--(9.8,-9);
				\draw[semithick][->,>=stealth](10.1,-9)--(10.5,-9);
				\draw[->>](8.9,-10.2)to[out=180,in=270](7.85,-9.15);
				\draw[->>](8.9,-10.2)to[out=0,in=270](9.95,-9.15);
				\draw[->>](8.9,-10.4)to[out=180,in=270](7.15,-9.15);
				\draw[->>](8.9,-10.4)to[out=0,in=270](10.65,-9.15);
				\node at(8.9,-10.6){$\nu$};
				\node at(-1.5,-12){$(G_2,2)$};
				\fill[black](0.2,-12)circle(1ex)node[below]{$-$};
				\fill[black](0.9,-12)circle(1ex)node[above left]{$+$};
				\fill[black](1.6,-12)circle(1ex)node[above left]{$-$};
				\fill[black](2.3,-12)circle(1ex)node[above right]{$+$};
				\fill[black](3,-12) circle(1ex)node[below]{$-$};
				\draw(1.6,-10.4)circle(1ex)node[above]{$\overset{\rm III}{}$};;
				\draw(0.216,-12.8)circle(1ex)node[below left]{$\overset{\rm V}{}$};
				\draw(2.984,-12.8)circle(1ex)node[below right]{$\overset{\rm I}{}$};
				\draw[semithick][->,>=stealth](0.9,-12)--(0.35,-12);
				\draw[semithick][->,>=stealth](0.9,-12)--(1.45,-12);
				\draw[semithick][->,>=stealth](2.3,-12)--(1.75,-12);
				\draw[semithick][->,>=stealth](2.3,-12)--(2.85,-12);
				\draw[semithick][->,>=stealth](1.6,-12)--(1.6,-10.55);
				\draw[semithick][->,>=stealth](1.55,-10.55)--(0.9,-11.85);
				\draw[semithick][->,>=stealth](1.65,-10.55)--(2.3,-11.85);
				\draw[semithick][->,>=stealth](0.2,-12)--(1.5,-10.5);
				\draw[semithick][->,>=stealth](3,-12)--(1.7,-10.5);
				\draw[semithick][->,>=stealth](1.6,-12)--(0.316,-12.7);
				\draw[semithick][->,>=stealth](0.266,-12.65)--(0.9,-12.15);
				\draw[semithick][->,>=stealth](0.366,-12.8)--(2.3,-12.15);
				\draw[-,line width=4pt,white](2.784,-12.6)--(1.8,-12.2);
				\draw[semithick][->,>=stealth](1.7,-12.1)--(2.884,-12.7);
				\draw[->>](2.884,-12.9)to[out=240,in=300](0.316,-12.9);
				\draw[->>](0.116,-12.7)to[out=120,in=180](1.45,-10.4);
				\draw[->>](1.75,-10.4)to[out=0,in=60](3.084,-12.7);
				\node at(1.6,-13.8){$\nu$};
		\end{tikzpicture}\vspace{-2mm}

 \caption{Dynkin quivers $Q(X_n,2)$.}
		\label{Dq}
	\end{figure}

	\subsection{Q-systems and Y-systems}
Then we review the notions of restricted Q-systems and restricted constant Y-systems.
Notations follow those used in \cite{KNS}.
Let $\Delta$ be a root system of type $X_n$ with a normalize inner product so that $\pair{\alpha,\alpha}=2$ for long roots $\alpha$.
Let $\Delta_+$, $\Delta_{{\rm long}}$ and $\Delta_{{\rm short}}$ be the set of positive roots, long roots and short loots, respectively.
For the simple root $\alpha_i$, we define a integer $t_i$ by
		\[t_i=\frac{2}{\pair{\alpha_i,\alpha_i}}\]
for $i=1,\dots,n$.
Set
	\[H=\{(i,m)\mid 1\leq i\leq n,\, 1\leq m\}\supset H_\ell=\{(i,m)\mid1\leq i\leq n,\, 1\leq m\leq t_i\ell-1\}.\]
Let $C=(C_{i,j})_{1\leq i,j\leq n}$ be the Cartan matrix of type $X_n$.
Then we define the number associated with the pair $(i,m),(j,k)\in H$ by
	\[G_{im,jk}=
		\begin{cases}
		-C_{j,i}(\delta_{m,2k-1}+2\delta_{m,2k}+\delta_{m,2k+1}),& t_i/t_j=2,\\
		-C_{j,i}(\delta_{m,3k-2}+2\delta_{m,3k-1}+3\delta_{m,3k}+2\delta_{m,3k+1}+\delta_{m,3k+2},& t_i/t_j=3,\\
		-C_{i,j}\delta_{t_jm,t_ik},& \text{otherwise}.
		\end{cases}\]
The unrestricted Q-system of type $X_n$ is the set of algebraic relations on the variables $\bigl\{Q_m^{(i)}\mid(i,m)\in H\bigr\}$ which has the form
	\[\bigl(Q_m^{(i)}\bigr)^2=Q_{m-1}^{(i)}Q_{m+1}^{(i)}+\bigl(Q_m^{(i)}\bigr)^2\prod_{(j,k)\in H}\bigl(Q_k^{(j)}\bigr)^{G_{im,jk}},\]
where \smash{$Q_m^{(0)}=Q_0^{(i)}=1$}.
By restricting the index set $H$ to $H_\ell$, we obtain the level $\ell$ restricted Q-system.
Namely, it is the algebraic relation among the variables \smash{$\big\{Q_m^{(i)}\mid(i,m)\in H_\ell\big\}$} such~that
	\[\bigl(Q_m^{(i)}\bigr)^2=Q_{m-1}^{(i)}Q_{m+1}^{(i)}+\bigl(Q_m^{(i)}\bigr)^2\prod_{(j,k)\in H_\ell}\bigl(Q_k^{(j)}\bigr)^{G_{im,jk}},\]
where \smash{$Q_m^{(0)}=Q_0^{(i)}=Q_{t_i\ell}^{(i)}=1$}.
Then the level $\ell$ restricted constant Y-system of type $X_n$ is the set of algebraic relations on \smash{$\bigl\{Y_m^{(i)}\mid(i,m)\in H_\ell\bigr\}$} which has the form
	\begin{align}\label{Y-sys}
	\bigl(Y_m^{(i)}\bigr)^2=\frac{\prod_{(j,k)\in H_\ell}\bigl(1+Y_k^{(j)}\bigr)^{G_{im,jk}+2\delta_{i,j}\delta_{m,k}}}{\bigl(1+\bigl(Y_{m-1}^{(i)}\bigr)^{-1}\bigr)\bigl(1+\bigl(Y_{m+1}^{(i)}\bigr)^{-1}\bigr)},
	\end{align}
where \smash{$\bigl(Y_0^{(i)}\bigr)^{-1}=\bigl(Y_{t_i\ell}^{(i)}\bigr)^{-1}=0$}.
Furthermore, the level $\ell$ restricted Q-system and the level $\ell$ restricted constant Y-system have the following relationship.
	\begin{Proposition}[{\cite[Proposition 14.1]{KNS}}]\label{QY}
	Suppose \smash{$Q_m^{(i)}$} satisfies the level $\ell$ restricted Q-system of type $X_n$.
	Then
		\[Y_m^{(i)}=\frac{\bigl(Q_m^{(i)}\bigr)^2\prod_{(j,k)\in H_\ell}\bigl(Q_k^{(j)}\bigr)^{G_{im,jk}}}{Q_{m-1}^{(i)}Q_{m+1}^{(i)}}\]
	is a solution of the level $\ell$ restricted constant Y-system of type $X_n$.
	\end{Proposition}
For restricted constant Y-systems, we know the following fact
	\begin{Theorem}[{\cite{IIKKN, IIKKN2, KN}}]\label{Ysol}
	There exists a unique solution of the level $\ell$ restricted constant Y-system for type $X_n$ satisfying \smash{$Y_m^{(i)}\in\R_{\geq0}$} for all $(i,m)\in H_\ell$.
	\end{Theorem}
	
In order to give the explicit solution of the Q-system of classical type $X_n$ characterized in the Theorem \ref{Ysol}, we introduce the notion of $q$-dimension.
Let $\geh$ be the simple Lie algebra of type~$X_n$.
Let $P=\Z\varpi_1\op\cdots\op\Z\varpi_n$ be the weight lattice of type $\geh$ and $P_+=\Z_{\geq0}\varpi_1\op\cdots\op\Z_{\geq0}\varpi_n$ be its subset.
Let $\chi(V_\varpi)$ be the character of the irreducible finite dimensional representation $V_\varpi$ of $\geh$ with highest weight $\varpi\in P_+$.
The specialization $\dim_q(V_\varpi)$ of $\chi(V_\varpi)$ is given by \[
\dim\nolimits_qV_\varpi=\prod_{\alpha\in\Delta_+}\frac{\sin{\frac{\pi\pair{\alpha,\varpi+\rho}}{\ell+h^\vee}}}{\sin{\frac{\pi\pair{\alpha,\rho}}{\ell+h^\vee}}},
\]
where $\rho$ is the half sum of positive roots.
This is the $q$-dimension at the root of unity \smash{$q={\rm e}^\frac{\pi {\rm i}}{t(\ell+h^\vee)}$}.
It is known that classical character of the Kirillov--Reshetikhin module \smash{$Q_m^{(i)}={\rm res}\chi_q\bigl(W_m^{(i)}\bigr)$} satisfies the unrestricted Q-system.
Furthermore, \smash{${\rm res}\chi_q\bigl(W_m^{(i)}\bigr)$} is expressed as a linear combination of certain characters $\chi(V_\varpi)$ for $\varpi\in P_+$.
For simplicity, denote $\chi(V_\varpi)$ by $\chi(\varpi)$.
Specifically, we~have
	\begin{align}\label{QA}
	Q_m^{(i)}=\chi(m\varpi_i)
	\end{align}
for $A_n$,
	\begin{align}
	&Q_m^{(i)}=\sum\chi(k_{i^\prime}\varpi_{i^\prime}+\cdots+k_{i-2}\varpi_{i-2}+k_i\varpi_i),\nonumber\\
 & \hphantom{Q_m^{(i)}=}  1\leq i\leq n,\ i^\prime\in\{0,1\},\ i^\prime\equiv i\pmod{2}\label{QB}
	\end{align}
for $B_n$, where $\varpi_0$ stands for 0, and the sum goes over nonnegative integers $k_{i^\prime},\dots,k_{i-2},k_i$ such that $t_i(k_{i^\prime}+\cdots+k_{i-2})+k_i=m$.
	\begin{align}\label{QC}
	Q_m^{(i)}=\begin{cases}
		\sum\chi(k_1\varpi_1+\cdots+k_i\varpi_i),&1\leq i\leq n-1,\\
		\chi(m\varpi_n),&i=n
	\end{cases}
	\end{align}
for $C_n$, where the sum goes over nonnegative integers $k_1,\dots,k_i$ such that $k_1+\cdots+k_i\leq m$ and $k_j\equiv m\delta_{i,j}\pmod{2}$ for all $1\leq j\leq i$,
	\begin{align}\label{QD}
\begin{aligned}
&Q_m^{(i)}=\sum\chi(k_{i^\prime}\varpi_{i^\prime}+\cdots+k_{i-2}\varpi_{i-2}+k_i\varpi_i),\\
 &\hphantom{Q_m^{(i)}=} 1\leq i\leq n-2,\ i^\prime\in\{0,1\},\ i^\prime\equiv i\pmod{2}, \\
	&Q_m^{(i)}=\chi(m\varpi_i),\quad i=n-1,n
\end{aligned}
	\end{align}
for $D_n$, where $\varpi_0$ stands for 0, and sum goes over nonnegative integers $k_{i^\prime},\dots,k_{i-2},k_i$ such~that $k_{i^\prime}+\cdots+k_{i-2}+k_i=m$.
Denote the specialization of \smash{${\rm res}\chi_q\bigl(W_m^{(i)}\bigr)$} to the $q$-dimension by \smash{$\dim_q{\rm res}\bigl(W_m^{(i)}\bigr)$}.
By the definition, \smash{$Q_m^{(i)}=\dim_q{\rm res}\bigl(W_m^{(i)}\bigr)$} still satisfies the unrestricted Q-system.
Then, for classical types $X_n$, we have the following truncation.
	\begin{Theorem}[{\cite{L,L2}}]\label{Qsol}
	\smash{$Q_m^{(i)}=\dim_q{\rm res}\bigl(W_m^{(i)}\bigr)$} satisfies the level $\ell$ restricted $Q$-system.
	Moreover strongly, the following properties hold for any $1\leq i\leq n$:
			\begin{enumerate}\itemsep=0pt
			\item[$1)$] $Q_m^{(i)}=Q_{t_i\ell-m}^{(i)}$ for $0\leq m\leq t_i\ell$,
			\item[$2)$] $Q_m^{(i)}<Q_{m+1}^{(i)}$ for $0\leq m\leq[t_i\ell/2]$,
			\item[$3)$] $Q_{t_i\ell+j}=0$ for $1\leq j\leq t_ih^\vee-1$,
			\end{enumerate}
	where $[t_i\ell/2]$ is the largest integer not exceeding $t_i\ell/2$.
	\end{Theorem}
Theorem \ref{Qsol} implies \smash{$Q_m^{(i)}>0$} for all $(i,m)\in H_\ell$.
In fact, by calculating characters $\chi(V_{\varpi})$, we have the following results for level 2 restricted Q-systems.
For $A_n$, by using (\ref{QA}), we have
	\begin{align*}
	Q_1^{(i)}=\prod_{j=1}^i\prod_{k=1}^{n+1-i}\frac{s(j+k)}{s(j+k-1)}
	\end{align*}
for $s(x)=\sin{\frac{\pi x}{n+3}}$.
For $B_n$, by using (\ref{QB}), we have
	\begin{align}\label{SQB}
	Q_1^{(i)}=i+1,\quad 1\leq i\leq n-1,\qquad Q_1^{(n)}=Q_3^{(n)}=\sqrt{2n+1},\qquad Q_2^{(n)}=n+1.
	\end{align}
For $C_n$, by using (\ref{QC}), we have
	\begin{align}\label{SQC}
	&Q_1^{(i)}=Q_3^{(i)}=\frac{s\bigl(\frac{i+1}{2}\bigr)s\bigl(\frac{i+3}{2}\bigr)s(i+2)}{s\bigl(\frac{1}{2}\bigr)s\bigl(\frac{3}{2}\bigr)s(2)},\qquad 1\leq i\leq n ,\\
	\label{SQC2}
	&Q_2^{(i)}=2\sum_{j=0}^i\frac{s(j)s(j+1)s(j+2)}{s(1)s(2)s(3)}+\frac{s(i+1)s(i+2)s(i+3)}{s(1)s(2)s(3)},\qquad 1\leq i\leq n-1,
	\end{align}
for $s(x)=\sin{\frac{\pi x}{n+3}}$.
For $D_n$, by using (\ref{QD}), we have
	\begin{align}\label{SQD}
	Q_1^{(i)}=i+1,\quad 1\leq i\leq n-2 ,\qquad Q_1^{(n-1)}=Q_1^{(n)}=\sqrt{n}.
	\end{align}
Therefore, \smash{$Y_m^{(i)}$} constructed by Proposition \ref{QY} with \smash{$Q_m^{(i)}=\dim_q{\rm res}\bigl(W_m^{(i)}\bigr)$} is real positive for all~$(i,m)\in H_\ell$.

	\section[Exponents of Q(X\_n,ell)]{Exponents of $\boldsymbol{Q(X_n,\ell)}$}
	
Set $Y=(Y_1,\dots,Y_N)$ as $(y_1,\dots,y_N)$ and denote by $y$.
We review the definition of exponents of~$Q(X_n,\ell)$ and verify their relationship with Y-systems.
	
	\subsection{Exponents}
For the rational function $\mu_\gamma(y)$, we have the following property.
		\begin{Proposition}[{\cite{M}}]\label{eta}
		The fixed point equation $\mu_\gamma(y)=y$ has a unique positive real solution.
		\end{Proposition}
Denote such a unique positive real solution by $\eta\in\R^N$.
Combining Theorem \ref{Ysol}, Proposition~\ref{eta} and the explicit form of $\mu_\gamma(y)$ calculated in the next section, we find that the solution~$\eta$ is given in terms of restricted constant Y-systems.
That is, $\eta$ is given to satisfy the relation~(\ref{Y-sys}).
Using Proposition \ref{eta}, we are able to define the unique matrix by
		\[J_\gamma(\eta)=\biggl(\frac{\partial\mu_\gamma(y)_i}{\partial y_j}\biggr)\biggr|_{y=\eta}\]
from the Jacobian of the rational function $\mu_\gamma(y)$, where $\mu_\gamma(y)_i$ is the $i$-th component of $\mu_\gamma$.
By the periodicity, we have $J_\gamma(\eta)^{t(\ell+h^\vee)}=I$.
Therefore, all eigenvalues of $J_\gamma(\eta)$ are $t(\ell+h^\vee)$-th root of unities.
Namely, these eigenvalues are written by
\[
{\rm e}^\frac{2\pi {\rm i}m_1}{t(\ell+h^\vee)},\dots,{\rm e}^\frac{2\pi {\rm i}m_N}{t(\ell+h^\vee)}
\]
by using a sequence of integers $0\leq m_1\leq\cdots\leq m_N<t(\ell+h^\vee)$.
We call this sequence $\mathcal{E}=(m_1,\dots,m_N)$ as the exponents of $Q(X_n,\ell)$.
		
	\subsection{Conjectural formulas}

In order to describe a conjectural formula associated with the type $X_n$, we define two polynomials
		\[N_{X_n,\ell}(z)=\prod_{i=1}^n\frac{z^{t(\ell+h^\vee)}-1}{z^{t/t_i}-1},\qquad D_{X_n,\ell}(z)=D_{X_n,\ell}^{{\rm long}}(z)D_{X_n,\ell}^{{\rm short}}(z),\]
where the polynomials \smash{$D_{X_n,\ell}^{{\rm long}}(z)$} and \smash{$D_{X_n,\ell}^{{\rm short}}(z)$} are defined by
		\[D_{X_n,\ell}^{{\rm long}}(z)=\prod_{\alpha\in\Delta_{{\rm long}}}\Bigl(z^t-{\rm e}^\frac{2\pi {\rm i}\pair{\rho,\alpha}}{\ell+h^\vee}\Bigr),\qquad D_{X_n,\ell}^{{\rm short}}(z)=\prod_{\alpha\in\Delta_{{\rm short}}}\Bigl(z-{\rm e}^\frac{2\pi {\rm i}\pair{\rho,\alpha}}{\ell+h^\vee}\Bigr).\]
The following conjecture is stated by Mizuno.
	\begin{Conjecture}[{\cite[Conjecture 3.8]{M}}]\label{MConj}
	Let $X_n$ be a finite type Dynkin diagram and $\ell$ be a~positive integer such that $\ell\geq2$.
	Let $\gamma=\gamma(X_n,\ell)$ be the mutation loop on the quiver $Q(X_n,\ell)$.
	Then the following identity holds:
		\[\det (zI-J_\gamma(\eta))=\frac{N_{X_n,\ell}(z)}{D_{X_n,\ell}(z)}.\]
	\end{Conjecture}
	\begin{Remark}
	This conjecture has been proven by Mizuno for $\gamma(A_1,\ell)$ and $\gamma(A_n,2)$.
	\end{Remark}
By this identity, we obtain a relationship with the theory of cluster algebras and the one of affine Lie algebras.
See \cite{M}.

	\section{Main results}
	From now, we assume that level $\ell$ is equal to 2.
	We treat the case that the rank of $X_n$ is even, since the argument for the odd rank is parallel to the one for the even rank.

	\subsection[Type B\_n]{Type $\boldsymbol{B_n}$}\label{section4.1}
Set $n$ as $2l$.
For the type $B_{2l}\ (l\geq1)$, set
\[
y=\bigl(y_1^{(1)},\dots,y_1^{(2l-1)},y_1^{(2l)},y_2^{(2l)},y_3^{(2l)},y_1^{(2l+1)},\dots, y_1^{(4l-1)}\bigr),
\]
 where the variable \smash{{$y_j^{(i)}$}} corresponds to the $i$-th
vertex from the left and the $j$-th vertex from the~top of $Q(B_{2l},2)$.
Now we have the cluster transformation as follows:
		\begin{align*}
		\left(\begin{array}{@{}l@{}}
			y_1^{(2i-1)}\\[1.3pt]
			y_1^{(2i)}\\[1.3pt]
			y_1^{(2l-1)}\\[1.3pt]
			y_{1,3}^{(2l)}\\[1.3pt]
			y_2^{(2l)}\\[1.3pt]
			y_1^{(2l+1)}\\[1.3pt]
			y_1^{(2l+2i)}\\[1.3pt]
			y_1^{(2l+2i+1)}
		\end{array} \right)&\overset{\mu_+}{\mapsto}
		\left( \begin{array}{@{}l@{}}
			y_1^{(2i-1)}\bigl(y_1^{(2i-2)}+1\bigr)\bigl(y_1^{(2i)}+1\bigr)\\[1.3pt]
			\bigl(y_1^{(2i)}\bigr)^{-1}\\[1.3pt]
			y_1^{(2l-1)}\bigl(y_1^{(2l-2)}+1\bigr)\bigl(y_1^{(2l)}+1\bigr)\bigl(y_3^{(2l)}+1\bigr)\\[1.3pt]
			\bigl(y_{1,3}^{(2l)}\bigr)^{-1}\\[1.3pt]
			y_2^{(2l)}\bigl(\bigl(y_1^{(2l)}\bigr)^{-1}+1\bigr)^{-1}\bigl(\bigl(y_3^{(2l)}\bigr)^{-1}+1\bigr)^{-1}\bigl(y_1^{(2l+1)}+1\bigr)\\[1.3pt]
			\bigl(y_1^{(2l+1)}\bigr)^{-1}\\[1.3pt]
			y_1^{(2l+2i)}\bigl(y_1^{(2l+2i-1)}+1\bigr)\bigl(y_1^{(2l+2i+1)}+1\bigr)\\[1.3pt]
			\bigl(y_1^{(2l+2i+1)}\bigr)^{-1}
		\end{array} \right)\\[1.3pt]
		&\overset{\mu_-}{\mapsto}
		\left( \begin{array}{@{}l@{}}
			\bigl(\mu_+(y)_1^{(2i-1)}\bigr)^{-1}\\[1.3pt]
			\mu_+(y)_1^{(2i)}\\[1.3pt]
			\mu_+(y)_1^{(2l-1)}\bigl(\mu_+(y)_2^{(2l)}+1\bigr)\\[1.3pt]
			\mu_+(y)_{1,3}^{(2l)}\bigl(\bigl(\mu_+(y)_2^{(2l)}\bigr)^{-1}+1\bigr)^{-1}\\[1.3pt]
			\bigl(\mu_+(y)_2^{(2l)}\bigr)^{-1}\\[1.3pt]
			\mu_+(y)_1^{(2l+1)}\bigl(\mu_+(y)_2^{(2l)}+1\bigr)\\[1.3pt]
			\bigl(\mu_+(y)_1^{(2l+2i)}\bigr)^{-1}\\[1.3pt]
			\mu_+(y)_1^{(2l+2i+1)}
		\end{array} \right)\overset{\nu}{\mapsto}\left( \begin{array}{@{}l@{}}
			\mu_-\mu_+(y)_1^{(4l-2i+1)}\\[1.3pt]
			\mu_-\mu_+(y)_1^{(4l-2i)}\\[1.3pt]
			\mu_-\mu_+(y)_1^{(2l+1)}\\[1.3pt]
			\mu_-\mu_+(y)_{1,3}^{(2l)}\\[1.3pt]
			\mu_-\mu_+(y)_2^{(2l)}\\[1.3pt]
			\mu_-\mu_+(y)_1^{(2l-1)}\\[1.3pt]
			\mu_-\mu_+(y)_1^{(2(l-i))}\\[1.3pt]
			\mu_-\mu_+(y)_1^{(2(l-i)-1)}
		\end{array} \right)
		\end{align*}
for $i=1,\dots,l-1$.
Here \smash{$y_1^{(0)}$} stands for 0.
For a solution of the level 2 restricted constant Y-system of type $B_{2l}$, if we take \smash{$\eta=\bigl(\eta_1^{(1)},\dots,\eta_1^{(2l-1)},\eta_1^{(2l)}, \eta_2^{(2l)},\eta_3^{(2l)},\eta_1^{(2l+1)},\dots,\eta_1^{(4l-1)}\bigr)$} to be
\begin{alignat*}{4}
&\eta_1^{(2i-1)}=\bigl(Y_1^{(2i-1)}\bigr)^{-1},\qquad&&\eta_1^{(2i)}=Y_1^{(2i)},\qquad&& 1\leq i\leq l-1, &\\
&\eta_1^{(2l-1)}=\bigl(Y_1^{(2l-1)}\bigr)^{-1}\bigl(Y_2^{(2l)}+1\bigr),\qquad&&\eta_{1,3}^{(2l)}=Y_{1,3}^{(2l)},\qquad&&\eta_2^{(2l)}=\bigl(Y_2^{(2l)}\bigr)^{-1},&\\
&\eta_1^{(2l+1)}=Y_1^{(2l-1)},\qquad&&\eta_1^{(j)}=\bigl(\eta_1^{(4l-j)}\bigr)^{-1},\qquad&& 2l+2\leq j\leq 4l-1,&
\end{alignat*}
then we find that $\eta$ is the solution of the fixed point equation $\mu_\gamma(y)=y$.
Furthermore, using Proposition~\ref{QY} and~(\ref{SQB}), we are able to calculate the value of \smash{$\bigl\{Y_m^{(i)}\mid(i,m)\in H_2\bigr\}$} satisfying~the~condition of Theorem \ref{Ysol} as follows:
\[Y_1^{(i)}=i(i+2),\quad 1\leq i\leq 2l-1 ,\qquad Y_1^{(2l)}=Y_3^{(2l)}=\frac{2l}{2l+1},\qquad Y_2^{(2l)}=\frac{4l^2}{4l+1}.\]

For a nonzero complex number $\lambda$, we consider the equations involving $(\phi_i)_{i=1}^{4l+1}$ such that
	\begin{align}\label{Beq}
	\sum_{j=1}^{4l+1}J_\gamma(\eta)_{i,j}\phi_j=\lambda\phi_i
	\end{align}
with the boundary conditions given by $\phi_0=\phi_{4l+2}=0$.
Upon calculating the Jacobian, we derive the relations
	\begin{align}
	\label{Br1}&\lambda\phi_{2k-1}=-\frac{\phi_{4l-2k+3}}{(4k^2-1)^2},\\
	\label{Br2}&\lambda\phi_{2k}=\frac{k\phi_{4l-2k+1}}{k+1}+16k^2(k+1)^2\phi_{4l-2k+2}+\frac{(k+1)\phi_{4l-2k+3}}{k},\\
	\label{Br3}&\lambda\phi_{2l+2k+2}=-\frac{\phi_{2l-2k}}{16(l-k)^2(l-k+1)^2},\\
&\lambda\phi_{2l+2k+3}=\frac{(2l-2k+1)\phi_{2l-2k-2}}{2l-2k-1}+(2l-2k-1)^2(2l-2k+1)^2\phi_{2l-2k-1}\nonumber\\
	&\hphantom{\lambda\phi_{2l+2k+3}=}{}+\frac{(2l-2k-1)\phi_{2l-2k}}{2l-2k+1}\label{Br4}
	\end{align}
for $k=1,\dots,l-1$ and
	\begin{align}
	\label{Br5}&\lambda\phi_{2l-1}=\frac{2l(2l+1)(\phi_{2l}+\phi_{2l+2})}{(2l-1)(4l+1)^2}+\frac{16l^4\phi_{2l+1}}{(4l^2-1)(4l+1)^2}-\frac{2\phi_{2l+3}}{(2l-1)^2(2l+1)(4l+1)},\\
	\label{Br6}&\lambda\phi_{2l,2l+2}=-\frac{2l\phi_{2l,2l+2}}{2l+1}+\frac{8l^3\phi_{2l+1}}{(2l+1)^3}+\frac{\phi_{2l+2,2l}}{2l+1}+\frac{(4l+1)\phi_{2+3}}{2l(2l+1)^3},\\
	\label{Br7}&(\lambda+1)\phi_{2l+1}=-\frac{(2l+1)^2(\phi_{2l}+\phi_{2l+2})}{8l^3}-\frac{(4l+1)\phi_{2l+3}}{16l^4},\\
&\lambda\phi_{2l+3}=\frac{(2l+1)\phi_{2l-2}}{2l-1}+(2l-1)^2(4l+1)\phi_{2l-1}+\bigl(4l^2-1\bigr)(\phi_{2l}+\phi_{2l+2})\nonumber\\
	&\hphantom{\lambda\phi_{2l+3}=}{}+\frac{16l^4(2l-1)\phi_{2l+1}}{(2l+1)(4l+1)}+\frac{(2l-1)\phi_{2l+3}}{2l+1}.\label{Br8}
	\end{align}
Then we have the following lemma.
		\begin{Lemma}\label{B}
		Set $\phi_{2l}=\phi_{2l+2}=1$.
		Then we have
			\begin{align*}
	&\phi_{2l-2k-1}=\frac{\Phi_{2l-2k-1}}{\lambda}\\
&\hphantom{\phi_{2l-2k-1}=}{}\times\!\bigl((2l-2k+1)\bigl(\lambda^{2k+1}+\lambda^{2k}+\lambda^{-2k}+\lambda^{-(2k+1)}\bigr) +2\bigl(\lambda^{2k-1}+\cdots+\lambda^{-(2k-1)}\bigr)\bigr),\\ &\phi_{2l-2k}=2\Phi_{2l-2k}\bigl((l-k+1)\bigl(\lambda^{2k}+\lambda^{2k-1}+\lambda^{-(2k-1)}+\lambda^{-2k}\bigr)+\lambda^{2(k-1)}+\cdots+\lambda^{-2(k-1)}\bigr),\\
			&\phi_{2l-1}=-\frac{2l(2l+1)^2}{(2l-1)^2(4l+1)^2}\bigl(2+(2l+1)\lambda^{-1}+(2l+1)\lambda^{-2}\bigr),\\
			&\phi_{2l+1}=-\frac{(2l+1)^3}{8l^3}\bigl(1+\lambda^{-1}\bigr),\\
			&\phi_{2l+3}=\frac{2l(2l+1)^2}{4l+1}\bigl((2l+1)\bigl(\lambda+\lambda^{-1}\bigr)+4n\bigr),\\
			&\phi_{2l+2k+2}=-\lambda(2l-2k-1)^2(2l-2k+1)^2\phi_{2l-2k},\\
			&\phi_{2l+2k+3}=-\frac{\phi_{2l-2k-1}}{16\lambda(l-k)^2(l-k+1)^2}
			\end{align*}
		for $k=1,\dots,l-1$, where $\Psi_m$ are rational numbers given by
			\begin{align*}
			&\Phi_{2l-2k-1}=-\frac{2(l-k)(2l+1)^2}{(2l-2k-1)^2(2l-2k+1)^2(4l+1)},\\
			&\Phi_{2l-2k}=\frac{(2l-2k+1)(2l+1)^2}{4l+1}.			
			\end{align*}
		Furthermore, such numbers satisfy the boundary condition if and only if $\lambda=\zeta^a$ for primitive $(4l+1)$-th root of unity $\zeta$ and $a=1,\dots,4l$.
		\end{Lemma}
		\begin{proof}
		The proof is given by direct calculation using relations (\ref{Br1})--(\ref{Br8}) under the assumption $\phi_{2l}=\phi_{2l+2}=1$.
		Then the boundary condition is provided by
			\[\phi_0=\frac{2(2+1)^2}{4l+1}\lambda^{-2l}\frac{\lambda^{4l+1}-1}{\lambda-1}=0.\]
		Therefore, we conclude that these numbers satisfy the boundary condition if and only if $\lambda=\zeta^a$ for $a=1,\dots,4l$.
		\end{proof}
		\begin{Remark}\label{spB}
		Now, the numbers $\phi=(0,\dots,0,\phi_{2l},0,\phi_{2l+2},0,\dots,0)$ determined from the tuple $(\lambda,\phi_{2l},\phi_{2l+2})=(-1,1,-1)$ also satisfies (\ref{Beq}) and the boundary condition.
		\end{Remark}
By Lemma \ref{B} and Remark \ref{spB}, it is understood that
	\[J_\gamma(\eta)\psi=\lambda\psi\]
when the vector $\psi=\bigl(\psi_1^{(1)},\dots,\psi_1^{(2l-1)},\psi_1^{(2l)},\psi_2^{(2l)},\psi_3^{(2l)},\psi_1^{(2l+1)},\dots,\psi_1^{(4l-1)}\bigr)$ is defined as
	\begin{gather*}
 \psi_1^{(i)}=\phi_i,\qquad 1\leq i\leq 2l-1,\\
 \psi_m^{(2l)}=\phi_{2l+m-1},\qquad m=1,2,3,\\
 \psi_1^{(j)}=\phi_{j+2},\qquad 2l+1\leq j\leq 4l-1.
 \end{gather*}

The quiver $Q(B_{2l-1},2)$, $l\geq2$, is obtained by removing one white vertex at each end of~$Q(B_{2l},2)$.
Note that we label the vertices from $1$ to $4l-3$ from left horizontally.
In view of this, the formulas of $\mu_+$ in the beginning of Section~\ref{section4.1}
for $B_{2l-1}$ are obtained by replacing~\smash{$y_m^{(2i+1)}$} \big(resp.\ \smash{$y_m^{(2i)}$}\big) with \smash{$y_m^{(2i)}$} \big(resp.\ \smash{$y_m^{(2i-1)}$}\big).
Hence, the calculations are parallel to the case of $B_{2l}$. Namely, we have
	\begin{align}\label{JB}
	\det (zI-J_\gamma(\eta))=(z+1)\frac{z^{2n+1}-1}{z-1}
	\end{align}
for the mutation loop $\gamma=\gamma(B_n,2)$.
		\begin{Theorem}
		For the mutation loop $\gamma=\gamma(B_n,2)$, we have
			\[\det (zI-J_\gamma(\eta))=\frac{N_{B_n,2}(z)}{D_{B_n,2}(z)}.\]
		\end{Theorem}
		\begin{proof}
		Let $\ep_1,\dots,\ep_n$ be the standard basis of the $n$-dimensional Euclidean space $\R^n$ so that $\pair{\ep_i,\ep_j}=\delta_{i,j}$.
		Since the root system of the type $B_n$ is given by
			\[\Delta=\{\pm\ep_i\pm\ep_j,\, 1\leq i<j\leq n,\, \pm\ep_i,\, 1\leq i\leq n\},\]
		we have
				\[N_{B_n,2}(z)=\frac{\bigl(z^{4n+2}-1\bigr)^n}{(z^2-1)^{n-1}(z-1)},\qquad D_{B_n,2}(z)=\biggl(\frac{z^{4n+2}-1}{z^2-1}\biggr)^{n-1}\cdot\prod_{\substack{1\leq k\leq 2n+1\\
			k\neq n+1}}\bigl(z-\zeta^\frac{2k-1}{2}\bigr).\]
		Therefore, the right-hand side becomes
			\[\frac{N_{B_n,2}(z)}{D_{B_n,2}(z)}=(z+1)\prod_{k=1}^{2n}\bigl(z-\zeta^k\bigr)=(z+1)\frac{z^{2n+1}-1}{z-1}\]
		and coincides with (\ref{JB}).
		\end{proof}

	\subsection[Type D\_n]{Type $\boldsymbol{D_n}$}
The argument for the type $D_n$ is similar to the type $B_n$.
Set $n$ as $2l$.
For the type $D_{2l}$, $l\geq2$, set $y=\bigl(y_1^{(1)},\dots,y_1^{(2l)}\bigr)$.
Now we have the cluster transformation as follows:
		\begin{align*}
		\left( \begin{array}{@{}l@{}}
			y_1^{(2i-1)}\\
			y_1^{(2i)}\\
			y_1^{(2l-2)}\\
			y_1^{(2l-1)}\\
			y_1^{(2l)}
		\end{array} \right)&\overset{\mu_+}{\mapsto}
		\left( \begin{array}{@{}l@{}}
			y_1^{(2i-1)}\bigl(y_1^{(2i-2)}+1\bigr)\bigl(y_1^{(2i)}+1\bigr)\\
			\bigl(y_1^{(2i)}\bigr)^{-1}\\
			\bigl(y_1^{(2l-2)}\bigr)^{-1}\\
			y_1^{(2l-1)}\bigl(y_1^{(2l-2)}+1\bigr)\\
			y_1^{(2l)}\bigl(y_1^{(2l-2)}+1\bigr)
		\end{array} \right)\\
		&\overset{\mu_-}{\mapsto}
		\left( \begin{array}{@{}l@{}}
			\bigl(\mu_+(y)_1^{(2i-1)}\bigr)^{-1}\\
			\mu_+(y)_1^{(2i)}\bigl(\mu_+(y)_1^{(2i-1)}+1\bigr)\bigl(\mu_+(y)_1^{(2i+1)}+1\bigr)\\
			\mu_+(y)_1^{(2l-2)}\bigl(\mu_+(y)_1^{(2l-3)}+1\bigr)\bigl(\mu_+(y)_1^{(2l-1)}+1\bigr)\bigl(\mu_+(y)_1^{(2l)}+1\bigr)\\
			\bigl(\mu_+(y)_1^{(2l-1)}\bigr)^{-1}\\
			\bigl(\mu_+(y)_1^{(2l)}\bigr)^{-1}
		\end{array} \right)
		\end{align*}
for $i=1,\dots,l-1$.
For a solution of the level 2 restricted constant Y-system of type $D_{2l}$, if we take \smash{$\eta=\bigl(\eta_1^{(1)},\dots,\eta_1^{(2l)}\bigr)$} to be
	\begin{alignat*}{4}
 &\eta_1^{(2i-1)}=\bigl(Y_1^{(2i-1)}\bigr)^{-1},\qquad&& \eta_1^{(2i)}=Y_1^{(2i)},\qquad&& 1\leq i\leq l-1, & \\ &\eta_1^{(2l-1)}=\bigl(Y_1^{(2l-1)}\bigr)^{-1},\qquad&&\eta_1^{(2l)}=\bigl(Y_1^{(2l)}\bigr)^{-1},\qquad&& &
 \end{alignat*}
then we find that $\eta$ is the solution of the fixed point equation $\mu_\gamma(y)=y$.
Now, using Proposition~\ref{QY} and~(\ref{SQD}), we are able to calculate the value of \smash{$\bigl\{Y_m^{(i)}\mid(i,m)\in H_2\bigr\}$} satisfying Theorem~\ref{Ysol} as follows:
	\[Y_1^{(i)}=i(i+2),\quad 1\leq i\leq 2l-2,\qquad Y_1^{(2l-1)}=Y_1^{(2l)}=2l-1.\]

For a nonzero complex number $\lambda$, we consider the equations involving $(\phi_i)_{i=1}^{2l}$ such that
	\begin{align}\label{Deq}
	\sum_{j=1}^{2l}J_\gamma(\eta)_{i,j}\phi_j=\lambda\phi_i
	\end{align}
with the boundary condition given by $\phi_0=0$.
In a similar manner to type $B_n$, we derive the relations
	\begin{align}
	\label{Dr1}&(\lambda+1)\phi_{2k-1}=-\frac{\phi_{2k-2}}{(2k-1)^3(2k+1)}-\frac{\phi_{2k}}{(2k-1)(2k+1)^3},\qquad 1\leq k\leq l-1,\\
&\lambda\phi_{2k}=\frac{(k+1)(2k+1)\phi_{2k-2}}{k(2k-1)}+\frac{(k+1)(2k-1)^2(2k+1)^2\phi_{2k-1}}{k}+\frac{(4k(k+1)-1)\phi_{2k}}{4k(k+1)}\nonumber\\
	&\hphantom{\lambda\phi_{2k}=}{}
 +\frac{k(2k+1)^2(2k+3)^2\phi_{2k+1}}{k+1}+\frac{k(2k+1)\phi_{2k+2}}{(k+1)(2k+3)},\qquad 1\leq k\leq l-2 ,\label{Dr2}\\
&\lambda\phi_{2l-2}=\frac{l(2l-1)\phi_{2l-4}}{(l-1)(2l-3)}+\frac{l(2l-3)^2(2l-1)^2\phi_{2l-3}}{l-1}+\frac{(2l-3)\phi_{2l-2}}{l-1}\nonumber\\
	&\hphantom{\lambda\phi_{2l-2}=}{}
 +2(l-1)(2l-1)^2(\phi_{2l-1}+\phi_{2l}),\label{Dr3}\\
	\label{Dr4}&(\lambda+1)\phi_{2l-1,2l}=-\frac{\phi_{2l-2}}{(2l-1)^3}.
	\end{align}
Then we have the following lemma.
		\begin{Lemma}\label{D}
		Set $\phi_{2l}=\phi_{2l-1}=1$.
		Then we have
			\begin{align*}
			&\phi_{2l-2k-1}=\Phi_{2l-2k-1}\bigl((2l-2k+1)\bigl(\lambda^k+\lambda^{-k}\bigr)+2\bigl(\lambda^{k-1}+\lambda^{-(k-1)}\bigr)+\cdots+2\bigr),\\
			&\phi_{2l-2k}=\Phi_{2l-2k}\bigl((l-k+1)\lambda^k+\lambda^{k-1}+\cdots+\lambda^{-(k-2)}+(l-k+1)\lambda^{-(k-1)}\bigr),
			\end{align*}
		 where $\Phi_m$ are rational numbers given by
			\begin{align*}
			&\Phi_{2l-2k-1}=\frac{(l-k)(2l-1)^2}{l(2l-2k-1)^2(2l-2k+1)^2},\\
			&\Phi_{2l-2k}=\frac{(2l-2k+1)(2l-1)^2}{l}.
			\end{align*}
		Furthermore, such numbers satisfy the boundary condition if and only if $\lambda=\zeta^a$ for primitive $2l$-th root of unity $\zeta$ and $a=1,\dots,2l-1$.
		\end{Lemma}
		\begin{proof}
		The proof is parallel to type $B_n$, since we have the relations (\ref{Dr1})--(\ref{Dr4}).
		Then the boundary condition is provided by
			\[\phi_0=\frac{2(2l-1)^2}{2l}\lambda^{-(l-1)}\frac{\lambda^{2l}-1}{\lambda-1}=0.\]
		Therefore, we conclude that these numbers satisfy the boundary condition if and only if $\lambda=\zeta^a$ for $a=1,\dots,2l-1$.
		\end{proof}
		\begin{Remark}\label{spD}
		Now, the numbers $\phi=(0,\dots,0,\phi_{2l-1},\phi_{2l})$ determined from the tuple
\[
(\lambda,\phi_{2l-1},\phi_{2l})=(-1,1,-1)
\]
 also satisfies (\ref{Deq}) and the boundary condition.
		\end{Remark}
By Lemma \ref{D} and Remark \ref{spD}, it is understood that
	\[J_\gamma(\eta)\psi=\lambda\psi\]
when we take the vector \smash{$\psi=\bigl(\psi_1^{(1)},\dots,\psi_1^{(2l)}\bigr)$} to be \smash{$\psi_1^{(i)}=\phi_i$} for $i=1,\dots,2l$.

The quiver $Q(D_{2l-1},2)$, $l\geq3$, is obtained by removing one black vertex at leftmost of $Q(D_{2l},2)$.
Hence in a similar manner to the case $B_n$, the calculations of $\gamma(D_{2l-1},2)$ are parallel to the case~$D_{2l}$.
Thus, we have
	\begin{align}\label{JD}
	\det (zI-J_\gamma(\eta))=(1+z)\frac{z^n-1}{z-1}
	\end{align}
for the mutation loop $\gamma=\gamma(D_n,2)$.
		\begin{Theorem}
For the mutation loop $\gamma=\gamma(D_n,2)$, we have
\[
\det (zI-J_\gamma(\eta))=\frac{N_{D_n,2}(z)}{D_{D_n,2}(z)}.
\]
		\end{Theorem}
		\begin{proof}
Since the root system of the type $D_n$ is given by
			\[\Delta=\{\pm\ep_i\pm\ep_j,\, 1\leq i<j\leq n\},\]
we have
			\[N_{D_n,2}(z)=\biggl(\frac{z^{2n}-1}{z-1}\biggr)^n,\qquad D_{D_n,2}(z)=\biggl(\frac{z^{2n}-1}{z-1}\biggr)^n\cdot\prod_{k=1}^{2n-1}\bigl(z-\zeta^k\bigr)^{-\theta(k)-\delta_{k,n}},\]
where $\theta(k)=1$ (resp.\ $\theta(k)=0$) if $k$ is even (resp.\ odd).
Therefore, the right-hand side becomes
			\[\frac{N_{D_n,2}(z)}{D_{D_n,2}(z)}=(z+1)\prod_{k=1}^{2n-1}\bigl(z-\zeta^k\bigr)^{\theta(k)}=(z+1)\frac{z^n-1}{z-1}\]
and coincides with (\ref{JD}).
		\end{proof}

\subsection[Type C\_n]{Type $\boldsymbol{C_n}$}

In a similar manner to the type $B_n$, set \smash{$y=\bigl(y_1^{(1)},y_2^{(1)},y_3^{(1)},\dots,y_3^{(2l-1)},y_1^{(2l)},y_1^{(2l+1)}\bigr)$} for the type~$C_{2l}$, $l\geq2$, where \smash{$y_1^{(2l)}$} \big(resp.\ \smash{$y_1^{(2l+1)}$}\big) corresponds to the white vertex with label $+$ (resp.~$-$).
Now we have the cluster transformation as follows:
		\begin{align*}
		\left( \begin{array}{@{}l@{}}
			y_{1,3}^{(2i-1)}\\
			y_2^{(2i-1)}\\
			y_{1,3}^{(2i)}\\
			y_2^{(2i)}\\
			y_{1,3}^{(2l-1)}\\
			y_2^{(2l-1)}\\
			y_1^{(2l)}\\
			y_1^{(2l+1)}
		\end{array} \right)&\overset{\mu_+}{\mapsto}
		\left( \begin{array}{@{}l@{}}
			\bigl(y_{1,3}^{(2i-1)}\bigr)^{-1}\\	y_2^{(2i-1)}\bigl(y_2^{(2i-2)}+1\bigr)\bigl(y_2^{(2i)}+1\bigr)\bigl(\bigl(y_1^{(2i-1)}\bigr)^{-1}+1\bigr)^{-1}\bigl(\bigl(y_3^{(2i-1)}\bigr)^{-1}+1\bigr)^{-1}\\
			y_{1,3}^{(2i)}\bigl(y_{1,3}^{(2i-1)}+1\bigr)\bigl(y_{1,3}^{(2i+1)}+1\bigr)\bigl(\bigl(y_2^{(2i)}\bigr)^{-1}+1\bigr)^{-1}\\
			\bigl(y_2^{(2i)}\bigr)^{-1}\\
			\bigl(y_{1,3}^{(2l-1)}\bigr)^{-1}\\		y_2^{(2l-1)}\bigl(y_2^{(2l-2)}+1\bigr)\bigl(y_1^{(2l)}+1\bigr)\bigl(\bigl(y_1^{(2l-1)}\bigr)^{-1}+1\bigr)^{-1}\bigl(\bigl(y_3^{(2l-1)}\bigr)^{-1}+1\bigr)^{-1}\\
			\bigl(y_1^{(2l)}\bigr)^{-1}\\
			y_1^{(2l+1)}\bigl(y_1^{(2l-1)}+1\bigr)\bigl(y_3^{(2l-1)}+1\bigr)
		\end{array} \right)\\
		&\overset{\mu_-}{\mapsto}
		\left( \begin{array}{@{}l@{}}	\mu_+(y)_{1,3}^{(2i-1)}\bigl(\mu_+(y)_{1,3}^{(2i-2)}+1\bigr)\bigl(\mu_+(y)_{1,3}^{(2i)}+1\bigr)\bigl(\bigl(\mu_+(y)_2^{(2i-1)}\bigr)^{-1}+1\bigr)^{-1}\\
			\bigl(\mu_+(y)_2^{(2i-1)}\bigr)^{-1}\\
			\bigl(\mu_+(y)_{1,3}^{(2i)}\bigr)^{-1}\\
			\mu_+(y)_2^{(2i)}\bigl(\mu_+(y)_2^{(2i-1)}+1\bigr)\bigl(\mu_+(y)_2^{(2i+1)}+1\bigr)\\
			 \ \times\bigl(\bigl(\mu_+(y)_1^{(2i)}\bigr)^{-1}+1\bigr)^{-1}\bigl(\bigl(\mu_+(y)_3^{(2i)}\bigr)^{-1}+1\bigr)^{-1} \\
			\mu_+(y)_{1,3}^{(2l-1)}\bigl(\mu_+(y)_{1,3}^{(2l-2)}+1\bigr)\bigl(\bigl(\mu_+(y)_2^{(2l-1)}\bigr)^{-1}+1\bigr)^{-1}\\
			\bigl(\mu_+(y)_2^{(2l-1)}\bigr)^{-1}\\
			\mu_+(y)_1^{(2l)}\bigl(\mu_+(y)_2^{(2l-1)}+1\bigr)\\
			\mu_+(y)_1^{(2l+1)}\bigl(\mu_+(y)_2^{(2l-1)}+1\bigr)
		\end{array} \right)\\
		&\overset{\nu}{\mapsto}
		\left( \begin{array}{@{}l@{}}
			\mu_-\mu_+(y)_{1,2,3}^{(j)}\\
			\mu_-\mu_+(y)_1^{(2l+1)}\\
			\mu_-\mu_+(y)_1^{(2l)}
		\end{array} \right)
		\end{align*}
for $i=1,\dots,l-1$ and $j=1,\dots, 2l-1$.
Here \smash{$y_{2}^{(0)}$} stands for 0.
For a solution of the level 2 restricted constant Y-system of type $C_{2l}$, if we take \smash{$\eta=\bigl(\eta_1^{(1)},\dots,\eta_3^{(2l-1)},\eta_1^{(2l)},\eta_1^{(2l+1)}\bigr)$} to be
	\begin{gather*}\eta_m^{(i)}=\begin{cases}
	Y_m^{(i)},&i+m\equiv0\pmod{2},\\
	\bigl(Y_m^{(i)}\bigr)^{-1},&i+m\equiv1\pmod{2},
	\end{cases}\qquad 1\leq i\leq2l-1,\ m=1,2,3 ,\\
	\eta_1^{(2l)}=Y_1^{(2l)},\qquad \eta_1^{(2l+1)}=\bigl(Y_1^{(2l)}\bigr)^{-1}\bigl(Y_2^{(2l-1)}+1\bigr),
 \end{gather*}
then we find that $\eta$ is the solution of the fixed point equation $\mu_\gamma(y)=y$.
Since the positive real solution of the restricted constant Y-system of type $C_{2l}$ with level 2 is complex (cf.\ (\ref{SQC}) and~(\ref{SQC2})), we will proceed the argument using the notation \smash{$\{Y_m^{(i)}\mid (i,m)\in H_2\}$}.	
	
We are able to decompose the Jacobian $J_\gamma(y)$ into
	\[J_\gamma(y)=(I-E_{6l-2,6l-2}-E_{6l-1,6l-1}+E_{6l-2,6l-1}+E_{6l-1,6l-1})J_\gamma^-(y^\prime)J_\gamma^+(y)\]
by the chain rule, where the leftmost matrix in the right-hand side expresses the action of $\nu$, $E_{i,j}$ is the matrix unit and $y^\prime=\mu_+(y)$.
Denote $\mu_+(\eta)$ by \smash{$\xi=\bigl(\xi_1^{(1)},\dots,\xi_3^{(2l-1)},\xi_1^{(2l)},\xi_1^{(2l+1)}\bigr)$.}
Then we have~\smash{$\xi_m^{(i)}=\bigl(\eta_m^{(i)}\bigr)^{-1}$.}

Let $v_1,\dots,v_{6l-1}$ be the standard basis of $\C^{6l-1}$.
For the vector
	\[\psi=\sum_{1\leq i\leq 2l-1,m=1,2,3}\psi_m^{(i)}v_{3(i-1)+m}+\psi_1^{(2l)}v_{6l-2}+\psi_1^{(2l+1)}v_{6l-1},\]
we set
	\begin{align*}
	&\psi^\prime=J_\gamma^+(\eta)\psi=\sum_{\substack{1\leq i\leq 2l-1\\ m=1,2,3}}{\psi_m^{(i)}}^\prime v_{3(i-1)+m}+{\psi_1^{(2l)}}^\prime v_{6l-2}+{\psi_1^{(2l+1)}}^\prime v_{6l-1},\\
	&\psi^{\prime\prime}=J_\gamma^-(\xi)\psi^\prime=\sum_{\substack{1\leq i\leq 2l-1\\ m=1,2,3}}{\psi_m^{(i)}}^{\prime\prime} v_{3(i-1)+m}+{\psi_1^{(2l)}}^{\prime\prime} v_{6l-2}+{\psi_1^{(2l+1)}}^{\prime\prime} v_{6l-1}.
	\end{align*}
	By direct calculation, we obtain the relations
		\begin{align}
		\label{Cr1}&{\psi_{1,3}^{(2k-1)}}^\prime=-\frac{\psi_{1,3}^{(1)}}{\bigl(Y_1^{(2k-1)}\bigr)^2},\\
		\label{Crp2}&{\psi_2^{(2k-1)}}^{\prime\prime}=-\frac{{\psi_2^{(2k-1)}}^\prime}{\bigl(Y_2^{(2k-1)}\bigr)^2}
		\end{align}
	for $k=1,\dots,l$, and
		\begin{align} \label{Cr2}&{\psi_2^{(2k-1)}}^\prime=Y_2^{(2k-1)}\\
\nonumber&\hphantom{{\psi_2^{(2k-1)}}^\prime=}{}\times
\Biggl(\frac{\psi_2^{(2k-2)}}{Y_2^{(2k-2)}+1}+\frac{\psi_1^{(2k-1)}+\psi_3^{(2k-1)}}{Y_1^{(2k-1)}\bigl(Y_1^{(2k-1)}+1\bigr)} +Y_2^{(2k-1)}\psi_2^{(2k-1)}+\frac{\psi_2^{(2k)}}{Y_2^{(2k)}+1}\Biggr),\\
		\label{Cr3}&{\psi_{1,3}^{(2k)}}^\prime=Y_1^{(2k)}\Biggl(\frac{\psi_{1,3}^{(2k-1)}}{Y_1^{(2k-1)}+1}+Y_1^{(2k)}\psi_{1,3}^{(2k)}+\frac{\psi_2^{(2k)}}{Y_2^{(2k)}\bigl(Y_2^{(2k)}+1\bigr)} +\frac{\psi_{1,3}^{(2k+1)}}{Y_1^{(2k+1)}+1}\Biggr),\\
		\label{Cr4}&{\psi_2^{(2k)}}^\prime=-\frac{\psi_2^{(2k)}}{\bigl(Y_2^{(2k)}\bigr)^2},\\
		\label{Crp1}&{\psi_{1,3}^{(2k-1)}}^{\prime\prime}=Y_1^{(2k-1)}\\
\nonumber&\hphantom{{\psi_{1,3}^{(2k-1)}}^{\prime\prime}=}{}\times
\Biggl(\frac{{\psi_{1,3}^{(2k-2)}}^\prime}{Y_2^{(2k-2)}+1}+Y_1^{(2k-1)}{\psi_{1,3}^{(2k-1)}}^\prime +\frac{{\psi_2^{(2k-1)}}^\prime}{Y_2^{(2k-1)}\bigl(Y_2^{(2k-1)}+1\bigr)}+\frac{{\psi_{1,3}^{(2k)}}^\prime}{Y_1^{(2k)}+1}\Biggr),\\
		\label{Crp3}&{\psi_{1,3}^{(2k)}}^{\prime\prime}=-\frac{{\psi_{1,3}^{(2k)}}^\prime}{\bigl(Y_1^{(2k)}\bigr)^2},\\
		\label{Crp4}&{\psi_2^{(2k)}}^{\prime\prime}=Y_2^{(2k)}\Biggl(\frac{{\psi_2^{(2k-1)}}^\prime}{Y_2^{(2k-1)}+1}+\frac{{\psi_1^{(2k)}}^\prime +{\psi_3^{(2k)}}^\prime}{Y_1^{(2k)}\bigl(Y_1^{(2k)}+1\bigr)}+Y_2^{(2k)}{\psi_2^{(2k)}}^\prime+\frac{{\psi_2^{(2k+1)}}^\prime}{Y_2^{(2k+1)}+1}\Biggr)
		\end{align}
	for $k=1,\dots,l-1$, and
		\begin{align}
		\label{Cr5}&{\psi_2^{(2l-1)}}^\prime=Y_2^{(2l-1)}\\
&\nonumber\hphantom{{\psi_2^{(2l-1)}}^\prime=}{}\times
\Biggl(\frac{\psi_2^{(2l-2)}}{Y_2^{(2l-2)}+1}+\frac{\psi_1^{(2l-1)}+\psi_3^{(2l-1)}}{Y_1^{(2l-1)}\bigl(Y_1^{(2l-1)}+1\bigr)}+Y_2^{(2l-1)}\psi_2^{(2l-1)}+\frac{\psi_1^{(2l)}}{Y_1^{(2l)}+1}\Biggr),\\
		\label{Crp5}&{\psi_{1,3}^{(2l-1)}}^{\prime\prime}=Y_1^{(2l-1)}\Biggl(\frac{{\psi_{1,3}^{(2l-2)}}^\prime}{Y_1^{(2l-2)}+1}+Y_1^{(2l-1)}{\psi_{1,3}^{(2l-1)}}^\prime+\frac{{\psi_2^{(2l-1)}}^\prime}{Y_2^{(2l-1)}\bigl(Y_2^{(2l-1)}+1\bigr)}\Biggr),\\
		\label{Cr6}&{\psi_1^{(2l)}}^{\prime\prime}=\frac{{\psi_2^{(2l-1)}}^\prime}{Y_1^{(2l)}}-\frac{\bigl(Y_2^{(2l-1)}+1\bigr)\psi_1^{(2l)}}{\bigl(Y_1^{(2l)}\bigr)^2},\\
		\label{Cr7}&{\psi_1^{(2l+1)}}^{\prime\prime}=Y_1^{(2l)}\Biggl(\frac{{\psi_2^{(2l-1)}}^\prime}{Y_2^{(2l-1)}+1}+\frac{\psi_1^{(2l-1)}+\psi_3^{(2l-1)}}{Y_1^{(2l-1)}+1}+\frac{Y_1^{(2l)}\psi_1^{(2l+1)}}{Y_2^{(2l-1)}+1}\Biggr).
		\end{align}

Set \smash{$u^{(k)}=v_{3k-2}-v_{3k}$} and \smash{$w^{(k)}=v_{3k-2}+v_{3k}$} for $k=1,\dots, 2l-1$.
We consider the subspaces
	\[U=\Op_{k=1}^{2l-1}\C u^{(k)},\qquad W=\Op_{k=1}^{2l-1}\C w^{(k)}\op\Op_{k=1}^{2l-1}\C v_{3k-1}\op\C v_{6l-2}\op\C v_{6l-1}.\]
Using (\ref{Cr1})--(\ref{Cr7}), we find that these subspace are invariant under the action of $J_\gamma(\eta)$.
Therefore, we have
	\begin{align*}
	J_\gamma(\eta){\bf u}={\bf u}
		\begin{pmatrix}
		\hat{K}&0\\
		0&\hat{L}
		\end{pmatrix}
	\end{align*}
for matrices $\hat{K}\in M_{2l-1}(\C)$ and $\hat{L}\in M_{4l}(\C)$, where \smash{${\bf u}=\bigl(u^{(1)},\dots,u^{(2l-1)},w^{(1)},v_2,\dots,w^{(2l-1)},$} $v_{6l-4},v_{6l-2},v_{6l-1}\bigr)$.

For simplicity, we define the notations \smash{$R_m^{(i)}$} and \smash{$S^{(i)}$} by
	\[R_m^{(i)}=Y_m^{(i)}Y_m^{(i+1)}\bigl(Y_m^{(i)}+1\bigr)^{-1}\bigl(Y_m^{(i+1)}+1\bigr)^{-1},\qquad S^{(i)}=2\bigl(Y_1^{(i)}+1\bigr)^{-1}\bigl(Y_2^{(i)}+1\bigr)^{-1}.\]
By (\ref{Cr1})--(\ref{Cr7}), we obtain the representation matrix $\hat{K}$ by
	\[\hat{K}_{i,j}=\begin{cases}
		-1,& j\equiv0\pmod{2},\ i=j, \\
		Y_1^{(i)}\bigl(Y_1^{(j)}\bigr)^2\bigl(Y_1^{(j)}+1\bigr)^{-1},& j\equiv0\pmod{2},\ i=j\pm1, \\
		-1+R_1^{(j-1)}+\bigl(1-\delta_{j,2l-1}\bigr)R_1^{(j)},& j\equiv1\pmod{2},\ i=j, \\
		-\bigl(Y_1^{(i)}\bigr)^{-1}\bigl(Y_1^{(j)}+1\bigr)^{-1},& j\equiv1\pmod{2},\ i=j\pm1, \\
		Y_1^{(i)}Y_1^{(j\pm1)}\bigl(Y_1^{(j)}+1\bigr)^{-1}\bigl(Y_1^{(j\pm1)}+1\bigr)^{-1},& j\equiv1\pmod{2},\ i=j\pm2,\\
		0,& \text{otherwise}
	\end{cases}\]
for $i,j=1,\dots,2l-1$.
By a similar calculation, we obtain the representation matrix $\hat{L}$ as follows:
	\begin{gather*}
 \hat{L}_{i,j}=\begin{cases}
		-1+R_2^{(\frac{j-2}{2})}+R_2^{(\frac{j}{2})}+S^{(\frac{j}{2})},& i=j,\\
		-\bigl(Y_1^{(\frac{j}{2})}\bigr)^{-1}\bigl(Y_2^{(\frac{j}{2})}\bigr)^{-1}\bigl(Y_2^{(\frac{j}{2})}+1\bigr)^{-1},& i=j-1,\\
		-\bigl(Y_2^{(\frac{i}{2})}\bigr)^{-1}\bigl(Y_2^{(\frac{j}{2})}+1\bigr)^{-1},& i=j\pm2,\\ Y_1^{(\frac{i+1}{2})}\bigl(Y_2^{(\frac{j}{2})}\!+1\bigr)^{-1}\bigl(\bigl(Y_2^{(\frac{i+1}{2})}\!+1\bigr)^{-1}+ Y_1^{(\frac{j}{2})}\bigl(Y_2^{(\frac{j}{2})}\bigr)^{-1}\bigl(Y_1^{(\frac{j}{2})}\!+1\bigr)^{-1}\bigr),\!& i+1=j\pm2,\\
		Y_2^{(\frac{j}{2}\pm1)}Y_2^{(\frac{i}{2})}\bigl(Y_2^{(\frac{j}{2})}+1\bigr)^{-1}\bigl(Y_2^{(\frac{j}{2}\pm1)}+1\bigr)^{-1},& i=j\pm4,
	\end{cases}
 \end{gather*}
for $i,j=1,\dots,4l-2$ such that $j\equiv0\pmod{4}$,
	\begin{gather*}\hat{L}_{i,j}=\begin{cases}
		-1+R_1^{(\frac{j-1}{2})}+(1-\delta_{j,4l-3})R_1^{(\frac{j+1}{2})}+S^{(\frac{j+1}{2})},&i=j,\\
		-2\bigl(Y_1^{(\frac{i}{2})}\bigr)^{-1}\bigl(Y_2^{(\frac{i}{2})}\bigr)^{-1}\bigl(Y_1^{(\frac{i}{2})}+1\bigr)^{-1},&i=j+1,\\
		-\bigl(Y_1^{(\frac{i+1}{2})}\bigr)^{-1}\bigl(Y_1^{(\frac{j+1}{2})}+1\bigr)^{-1},&i=j\pm2,\\
2Y_2^{(\frac{i}{2})}\bigl(Y_1^{(\frac{j+1}{2})}+1\bigr)^{-1} \\
\ \times \bigl(\bigl(Y_1^{(\frac{i}{2})}+1\bigr)^{-1}+Y_2^{(\frac{j+1}{2})}\bigl(Y_1^{(\frac{j+1}{2})}\bigr)^{-1}\bigl(Y_2^{(\frac{j+1}{2})}+1\bigr)^{-1}\bigr),&i-1=j\pm2,\\
		Y_1^{(\frac{j+1}{2}\pm1)}Y_1^{(\frac{i+1}{2})}\bigl(Y_1^{(\frac{j+1}{2})}+1\bigr)^{-1}\bigl(Y_1^{(\frac{j+1}{2}\pm1)}+1\bigr)^{-1},& i=j\pm4
	\end{cases}\end{gather*}
for $i,j=1,\dots,4l-2$ such that $j\equiv1\pmod{4}$,
	\[\hat{L}_{i,j}=\begin{cases}
		-1,& i=j,\\
		Y_1^{(\frac{j}{2})}Y_2^{(\frac{j}{2})}\bigl(Y_2^{(\frac{j}{2})}+1\bigr)^{-1},& i=j-1,\\
		Y_2^{(\frac{i}{2})}\bigl(Y_2^{(\frac{j}{2})}\bigr)^2\bigl(Y_2^{(\frac{j}{2})}+1\bigr)^{-1},& i=j\pm2,
	\end{cases}\]
for $i,j=1,\dots,4l-2$ such that $j\equiv2\pmod{4}$,
	\[\hat{L}_{i,j}=\begin{cases}
		-1,& i=j,\\
		2Y_1^{(\frac{i}{2})}Y_2^{(\frac{i}{2})}\bigl(Y_1^{(\frac{i}{2})}+1\bigr)^{-1},& i=j+1,\\
		Y_1^{(\frac{i+1}{2})}\bigl(Y_1^{(\frac{j+1}{2})}\bigr)^2\bigl(Y_1^{(\frac{j+1}{2})}+1\bigr)^{-1},& i=j\pm2
	\end{cases}\]
for $i,j=1,\dots,4l-2$ such that $j\equiv3\pmod{4}$ and
	\begin{align*}
		&\hat{L}_{4l-1,4l-4}=Y_2^{(2l-1)}Y_1^{(2l)}\bigl(Y_2^{(2l-2)}+1\bigr)^{-1}\bigl(Y_2^{(2l-1)}+1\bigr)^{-1},\\
		&\hat{L}_{4l,4l-4}=Y_2^{(2l-1)}\bigl(Y_1^{(2l)}\bigr)^{-1}\bigl(Y_2^{(2l-2)}+1\bigr)^{-1},\\
		&\hat{L}_{4l-1,4l-3}=2Y_1^{(2l)}\bigl(Y_1^{(2l-1)}+1\bigr)^{-1}\bigl(1+Y_2^{(2l-1)}(Y_1^{(2l-1)}\bigr)^{-1}\bigl(Y_2^{(2l-1)}+1\bigr)^{-1}\bigr),\\
		&\hat{L}_{4l,4l-3}=2Y_2^{(2l-1)}\bigl(Y_1^{(2l-1)}\bigr)^{-1}\bigl(Y_1^{(2l)}\bigr)^{-1}\bigl(Y_1^{(2l-1)}+1\bigr)^{-1},\\
		&\hat{L}_{4l-1,4l-2}=\bigl(Y_2^{(2l-1)}\bigr)^2Y_1^{(2l)}\bigl(Y_2^{(2l-1)}+1\bigr)^{-1},\\
		&\hat{L}_{4l,4l-2}=\bigl(Y_2^{(2l-1)}\bigr)^2\bigl(Y_1^{(2l)}\bigr)^{-1},\\
		&\hat{L}_{4l-4,4l-1}=Y_2^{(2l-2)}Y_2^{(2l-1)}\bigl(Y_2^{(2l-1)}+1\bigr)^{-1}\bigl(Y_1^{(2l)}+1\bigr)^{-1},\\
		&\hat{L}_{4l-3,4l-1}=Y_1^{(2l-1)}\bigl(Y_2^{(2l-1)}+1\bigr)^{-1}\bigl(Y_1^{(2l)}+1\bigr)^{-1},\\
		&\hat{L}_{4l-2,4l-1}=-\bigl(Y_2^{(2l-1)}\bigr)^{-1}\bigl(Y_1^{(2l)}+1\bigr)^{-1},\\
		&\hat{L}_{4l-1,4l-1}=Y_2^{(2l-1)}Y_1^{(2l)}\bigl(Y_2^{(2l-1)}+1\bigr)^{-1}\bigl(Y_1^{(2l)}+1\bigr)^{-1},\\
		&\hat{L}_{4l,4l-1}=Y_2^{(2l-1)}\bigl(Y_1^{(2l)}\bigr)^{-1}\bigl(Y_1^{(2l)}+1\bigr)^{-1}-\bigl(Y_2^{(2l-1)}+1\bigr)\bigl(Y_1^{(2l)}\bigr)^{-2},\\
		&\hat{L}_{4l-1,4l}=\bigl(Y_1^{(2l)}\bigr)^2\bigl(Y_2^{(2l-1)}+1\bigr)^{-1},\\
		&\hat{L}_{i,j}=0,\qquad \text{otherwise}.
	\end{align*}
	
Next we consider the matrix $\hat{K}-\lambda^2I$.
We are able to simplify the matrix $\hat{K}-\lambda^2I$ by the row or column operations taking the following steps:
	\begin{enumerate}\itemsep=0pt
		\item Add \smash{$Y_1^{(2k-1)}\bigl(Y_1^{(2k)}\bigr)^2\bigl(Y_1^{(2k)}+1\bigr)^{-1}$} times of the $2k$-th row to the $(2k-1)$-th row and $\bigl(Y_1^{(2k)}\bigr)^2Y_1^{(2k+1)}\bigl(Y_1^{(2k)}\!+1\bigr)^{-1}$ times of the $2k$-th row to the $(2k+1)$-th row for $k=1,\dots,l-1$.
		\item Multiply the $(2k-1)$-th row by $-\lambda^{-1}$ for $k=1,\dots,l$ and the $2k$-th row by \smash{$-\bigl(Y_1^{(2k)}\bigr)^2$} for $k=1,\dots,l-1$.
		\item Multiply the $2k$-th column by \smash{$\lambda^{-1}\bigl(Y_1^{(2k)}\bigr)^{-2}$} for $k=1,\dots,{l-1}$.
	\end{enumerate}
Denote the matrix obtained by these operations by $K$.
For the matrix $\hat{L}-\lambda^2I$, we apply similar operations.
Namely, the matrix $\hat{L}-\lambda^2I$ is simplified to the matrix $L$ by the following steps:
	\begin{enumerate}\itemsep=0pt
		\item Add \smash{$\bigl(Y_1^{(2k)}\bigr)^2Y_1^{(2k+1)}\bigl(Y_1^{(2k)}+1\bigr)^{-1}$} times of the $(4k-1)$-th row to the $(4k+1)$-th row, $Y_1^{(2k-1)}\bigl(Y_1^{(2k)}\bigr)^2\bigl(Y_1^{(2k)}+1\bigr)^{-1}$ times of the $(4k-1)$-th row to the $(4k-3)$-th row for $k=1,\dots,l-1$ and \smash{$Y_1^{(2k-1)}Y_2^{(2k-1)}\bigl(Y_2^{(2k-1)}+1\bigr)^{-1}$} times of the $(4k-2)$-th row to the $(4k-3)$-th row for $k=1,\dots,l$.
		
\item Add \smash{$\bigl(Y_2^{(2k-1)}\bigr)^2Y_2^{(2k)}\bigl(Y_2^{(2k-1)}+1\bigr)^{-1}$} times of the $(4k-2)$-th row, \smash{$2Y_1^{(2k)}Y_2^{(2k)}\bigl(Y_1^{(2k)}+1\bigr)^{-1}$} times of the $(4k-1)$-th row and \smash{$Y_2^{(2k)}\bigl(Y_2^{(2k+1)}\bigr)^2\bigl(Y_2^{(2k+1)}+1\bigr)^{-1}$} times of the $(4k+2)$-th row to the $4k$-th row for $k=1,\dots, l-1$.
		
\item Add \smash{$\bigl(Y_2^{(2l-1)}\bigr)^2Y_1^{(2l)}\bigl(Y_2^{(2l-1)}+1\bigr)^{-1}$} times of the $(4l-2)$-th row to the $(4l-1)$-th row and $\bigl(Y_2^{(2l-1)}\bigr)^2\bigl(Y_1^{(2l)}\bigr)^{-1}$ times of the $(4l-2)$-th row to the $4l$-th row.
		
\item Add \smash{$\bigl(Y_1^{(2l)}\bigr)^2\bigl(Y_2^{(2l-1)}+1\bigr)^{-1}$} times of the $4l$-th row (resp.\ $(4l-1)$-th column) to the $(4l-1)$-th row (resp.\ $4l$-th column).
		
\item Multiply the $(4k-3)$-th row by $-\lambda^{-1}$, the $(4k-2)$-th row by \smash{$-\bigl(Y_2^{(2k-1)}\bigr)^2$} for $k=1,\dots,l$, the $(4k-1)$-th row by \smash{$-\bigl(Y_1^{(2k)}\bigr)^2$}, the $4k$-th row by $-\lambda^{-1}$ for $k=1,\dots,l-1$, and the~${(4l-1)}$-th row by $-\lambda^{-1}$, the $4l$-th row by \smash{$-\bigl(Y_1^{(2l)}\bigr)^2$}.
		
\item Multiply the $(4k-2)$-th column by \smash{$\lambda^{-1}\bigl(Y_2^{(2k-1)}\bigr)^{-2}$} for $k=1,\dots,l$, the $(4k-1)$-th column by \smash{$\lambda^{-1}\bigl(Y_1^{(2k)}\bigr)^{-2}$} for $k=1,\dots,l-1$, and the $4l$-th column by \smash{$\lambda^{-1}\bigl(Y_1^{(2l)}\bigr)^{-2}$}.
	\end{enumerate}
	
The quiver $Q(C_{2l-1},2)$, $l\geq2$, is obtained by removing the three leftmost black vertices from~$Q(C_{2l},2)$.
Hence, in a similar manner to the case $B_n$, the calculations of $\gamma(C_{2l-1},2)$ are parallel to the case $C_{2l}$.

From now, we consider the mutation loop $\gamma=\gamma(C_n,2)$ and variables $y=\bigl(\smash{y_1^{(1)},\dots,y_3^{(n-1)}},\allowbreak \smash{y_1^{(n)},y_1^{(n+1)}}\bigr)$.
The matrices $K\in M_{n-1}(\C)$ and $L\in M_{2n}(\C)$ are given as follows:
	\[K_{i,j}=\begin{cases}
	\La,& j=i,\\
	Y_1^{(i)}\bigl(Y_1^{(j)}+1\bigr)^{-1},& j=i\pm1 ,\\
	0,& \text{otherwise}
	\end{cases}\]
for $i,j=1,\dots,n-1$,
	\[L_{i,j}=\begin{cases}
		\La,& j=i,\\
		Y_1^{(\frac{j}{2})}\bigl(Y_2^{(\frac{j}{2})}\bigr)^{-1}\bigl(Y_2^{(\frac{j}{2})}+1\bigr)^{-1},& i\equiv1\pmod{2},\ j=i+1,\\
		Y_1^{(\frac{i+1}{2})}\bigl(Y_1^{(\frac{j+1}{2})}+1\bigr)^{-1},& i\equiv1\pmod{2},\ j=i\pm2,\\
		2Y_2^{(\frac{i}{2})}\bigl(Y_1^{(\frac{i}{2})}\bigr)^{-1}\bigl(Y_1^{(\frac{i}{2})}+1\bigr)^{-1},& i\equiv0\pmod{2},\ j=i-1,\\
		Y_2^{(\frac{i}{2})}\bigl(Y_2^{(\frac{j}{2})}+1\bigr)^{-1},& i\equiv0\pmod{2},\ j=i\pm2,\\
		0,& \text{otherwise}
	\end{cases}\]
for $i,j=1,\dots,2n-2$ and
	\begin{align*}
	&L_{2n-1,2n-3}=-2\lambda^{-1}Y_1^{(n)}\bigl(Y_1^{(n-1)}+1\bigr)^{-1},\\
	&L_{2n-1,2n-2}=2Y_1^{(n)}\bigl(Y_2^{(n-1)}+1\bigr)^{-1},\\
	&L_{2n,2n-2}=\lambda Y_1^{(n)},\\
	&L_{2n-2,2n-1}=Y_2^{(n-1)}\bigl(Y_1^{(n)}+1\bigr)^{-1},\\
	&L_{2n-1,2n-1}=L_{2n,2n}=\Lambda,\\
	&L_{2n,2n-1}=Y_2^{(n-1)}+1,\\
	&L_{2n-2,2n}=\lambda^{-1}Y_2^{(n-1)}\bigl(Y_2^{(n-1)}+1\bigr)^{-1}\bigl(Y_1^{(n)}+1\bigr)^{-1},\\
	&L_{2n-1,2n}=\bigl(Y_2^{(n-1)}+1\bigr)^{-1},\\
	&L_{i,j}=0,\qquad \text{otherwise} ,
	\end{align*}
where $\Lambda=\lambda+\lambda^{-1}$.
Now, we have the following conjecture.
	\begin{Conjecture}\label{Csol}
	We have
		\begin{enumerate}\itemsep=0pt
		\item[$1)$] $\det (K)=\prod_{i=1}^{n-1}\Bigl(\Lambda-2\cos{\frac{(2i+3)\pi}{4(n+3)}}\Bigr)$,
		\item[$2)$] $\det (L)=\prod_{i=1}^{n-2}\Bigl(\La-2\cos{\frac{(i+2)\pi}{2(n+3)}}\Bigr)^2\prod_{j\in\{1,2,n+1,n+2\}}\Bigl(\La-2\cos{\frac{j\pi}{2(n+3)}}\Bigr)$.
		\end{enumerate}
	\end{Conjecture}
Under the assumption that Conjecture \ref{Csol} is true, we have the following theorem.
	\begin{Theorem}
	If Conjecture~$\ref{Csol}$ holds, we have
		\[\det (zI-J_\gamma(\eta))=\frac{N_{C_n,2}(z)}{D_{C_n,2}(z)}\]
	for the mutation loop $\gamma=\gamma(C_n,2)$.
	\end{Theorem}
	\begin{proof}
	For the matrices $\hat{K}$, $\hat{L}$, $K$ and $L$, we have the relations
		\[(-\lambda)^{-(n-1)}\det \bigl(\hat{K}-\lambda^2I\bigr)=\det (K),\qquad\lambda^{-2n}\det \bigl(\hat{L}-\lambda^2I\bigr)=\det (L).\]
	Therefore, by Conjecture \ref{Csol}, we have
		\begin{align*}
		\det (zI-J_\gamma(\eta))&=\det \left(\begin{pmatrix}
			zI-\hat{K}&0\\
			0&zI-\hat{L}\end{pmatrix}\right)\\
		&=z^\frac{3n-1}{2}\det (K)\det (L)\\
		&=\prod_k\bigl(z-\zeta^k\bigr),
		\end{align*}
	where $\zeta$ is a primitive $2n+6$-th root of unity and $k$ takes a value in the sequence $\mathcal{E}=(2,4,5,6,6,7,8,8,\dots,2n-1,2n,2n,2n+1,2n+2,2n+4)$.
	Here $\lambda$ stands for \smash{$z^\frac{1}{2}$} in Conjecture~\ref{Csol}.
	On the other hand, since the root system of the type $C_n$ is given by
		\[\Delta=\biggl\{\frac{1}{\sqrt{2}}(\pm\ep_i\pm\ep_j),\, 1\leq i<j\leq n,\, \pm\sqrt{2}\ep_i,\, 1\leq i\leq n\biggr\},\]
	we have
		\begin{gather*}
N_{C_n,2}(z)=\frac{\bigl(z^{2n+6}-1\bigr)^n}{(z^2-1)(z-1)^{n-1}},\\
 D_{C_n,2}(z)=\biggl(\frac{z^{n+3}+1}{z+1}\biggr)\cdot\prod_{1\leq k\leq 2n+5}\bigl(z-\zeta^k\bigr)^{n-2}\prod_{1\leq k\leq4}\bigl(z-\zeta^{\pm k}\bigr).
 \end{gather*}
	Therefore, we find that
		\[\frac{N_{C_n,2}(z)}{D_{C_n,2}(z)}=\biggl(\frac{z^{n+3}-1}{z-1}\biggr)\cdot\prod_{5\leq k\leq 2n+1}\bigl(z-\zeta^k\bigr)\]
	Consequently, we obtain the claim under the assumption.
	\end{proof}
	\begin{Example}
Set $Z$ as \smash{$z^\frac{1}{2}+z^{-\frac{1}{2}}$}.
In fact, Conjecture \ref{Csol} makes it possible for us to calculate the left-hand side of Mizuno's conjecture more easily than with the original $J_\gamma(\eta)$.
		\begin{enumerate}\itemsep=0pt
			\item $\gamma(C_3,2)$:
				\begin{align*}
				\det (zI-J_\gamma(\eta))&{}=z^4\det \left(\begin{pmatrix}
					Z&K_{1,2}\\
					K_{2,1}&Z
				\end{pmatrix}\right)\\
				&\quad{}\times\det \left(\begin{pmatrix}
					Z&L_{1,2}&L_{1,3}&0&0&0\\
					L_{2,1}&Z&0&L_{2,4}&0&0\\
					L_{3,1}&0&Z&L_{3,4}&0&0\\
					0&L_{4,2}&L_{4,3}&Z&L_{4,5}&L_{4,6}\\
					0&0&L_{5,3}&L_{5,4}&Z&L_{5,6}\\
					0&0&0&L_{6,4}&L_{6,5}&Z
				\end{pmatrix}\right).
				\end{align*}
			\item $\gamma(C_4,2)$:
				\begin{align*}
				\det (zI-J_\gamma(\eta))&{}=z^\frac{11}{2}\det \left(\begin{pmatrix}
					Z&K_{1,2}&0\\
					K_{2,1}&Z&K_{2,3}\\
					0&K_{3,2}&Z
				\end{pmatrix}\right)\\
				&\quad{}\times\det \left(\begin{pmatrix}
					Z&L_{1,2}&L_{1,3}&0&0&0&0&0\\
					L_{2,1}&Z&0&L_{2,4}&0&0&0&0\\
					L_{3,1}&0&Z&L_{3,4}&L_{3,5}&0&0&0\\
					0&L_{4,2}&L_{4,3}&Z&0&L_{4,6}&0&0\\
					0&0&L_{5,3}&0&Z&L_{5,6}&0&0\\
					0&0&0&L_{6,4}&L_{6,5}&Z&L_{6,7}&L_{6,8}\\
					0&0&0&0&L_{7,5}&L_{7,6}&Z&L_{7,8}\\
					0&0&0&0&0&L_{8,6}&L_{8,7}&Z
				\end{pmatrix}\right).
				\end{align*}
		\end{enumerate}
	\end{Example}

\subsection*{Acknowledgements}

The author would like to thank his supervisor Masato Okado for helpful comments and discussion on this research.
The author is also grateful to the anonymous referees for their valuable comments and suggestions, which helped to improve the clarity and quality of this paper.
The author was supported by JST, the establishment of university fellowships towards the creation of science technology innovation, Grant Number JPMJFS2138.
This work was partly supported by MEXT Promotion of Distinctive Joint Research Center Program JPMXP0723833165.

\pdfbookmark[1]{References}{ref}
\LastPageEnding

\end{document}